\documentclass[smallextended,referee,envcountsect]{svjour3}
\smartqed
\usepackage{graphicx}
\usepackage{xcolor}
\usepackage{bbm}
\usepackage{amsmath}
\usepackage[ruled,vlined]{algorithm2e}
\usepackage{algorithmicx}
\usepackage{float}
\usepackage{enumitem}
\usepackage{longtable}
\usepackage{caption}
\usepackage{lscape}
\usepackage{array}

\definecolor{verdeoscuro}{rgb}{0.0, 0.5, 0.0}

\begin{document}

\title{Enhancing sharp augmented Lagrangian methods with smoothing techniques for nonlinear programming}

% OTRAS OPCIONES
% A smoothing--based approach to nonlinear programming with sharp augmented Lagrangians
% Novel smoothing techniques for sharp augmented Lagrangians: a comparative study with PHR methods

\author{José Luis Romero \and Damián Fernández \and Germán Ariel Torres}

\institute{
José Luis Romero \and Damián Fernández \at
Facultad de Matemática, Astronomía, Física y Computación \\
Universidad Nacional de Córdoba (CIEM--CONICET) \\
Av. Medina Allende s/n, Ciudad Universitaria, CP X5000HUA, Córdoba, Argentina \\
e-mail: joseluisromero@unc.edu.ar
\and Damián Fernández \at 
e-mail: dfernandez@unc.edu.ar 
\and  
Germán Ariel Torres, Corresponding author \at
Facultad de Ciencias Exactas y Naturales y Agrimensura \\
Universidad Nacional del Nordeste (IMIT--CONICET) \\
Av. Libertad 5470, Corrientes, CP 3400, Corrientes, Argentina \\
e-mail: german.torres@comunidad.unne.edu.ar
}

\date{Received: date / Accepted: date}
%The correct dates will be entered by the editor.

\maketitle

\begin{abstract}
This paper proposes a novel approach to solving nonlinear programming problems using a sharp augmented Lagrangian method with a smoothing technique. Traditional sharp augmented Lagrangian methods are known for their effectiveness but are often hindered by the need for global minimization of nonconvex, nondifferentiable functions at each iteration. To address this challenge, we introduce a smoothing function that approximates the sharp augmented Lagrangian, enabling the use of primal minimization strategies similar to those in Powell--Hestenes--Rockafellar (PHR) methods. Our approach retains the theoretical rigor of classical duality schemes while allowing for the use of stationary points in the primal optimization process. We present two algorithms based on this method--one utilizing standard descent and the other employing coordinate descent. Numerical experiments demonstrate that our smoothing--based method compares favorably with the PHR augmented Lagrangian approach, offering both robustness and practical efficiency. The proposed method is particularly advantageous in scenarios where exact minimization is computationally infeasible, providing a balance between theoretical precision and computational tractability.
\end{abstract}
% 49: Calculus of variations and optimal control; optimization
% 49J53: Set-valued and variational analysis
% 49K99: None of the above, but in this section
% 90: Operations research, mathematical programming
% 90C30: Nonlinear programming

\section{Introduction}
We want to solve the following nonlinear programming problem:
\begin{equation} \label{eq:nlp}
\begin{array}{cl}
\textrm{minimize} & f \left( x \right) \\[2mm]
\textrm{subject to}& h \left( x \right) = 0 ,
\end{array}
\end{equation}
where $f : \mathbbm{R}^n \to \mathbbm{R}$ and $h : \mathbbm{R}^n \to \mathbbm{R}^m$ are twice continuously differentiable. Stationary points $x$ of \eqref{eq:nlp} along with the associated Lagrange multipliers $\lambda$, are characterized by the following system of equations
\begin{equation} \label{eq:lagrange}
\nabla_x L \left( x , \lambda \right) = 0 , \qquad h \left( x \right) = 0 ,
\end{equation}
where $L: \mathbbm{R}^n \times \mathbbm{R}^m \to \mathbbm{R}$ such that $L \left( x , \lambda \right) = f \left( x \right) + \left\langle \lambda , h \left( x \right) \right\rangle$, is the  Lagrangian function associated with problem \eqref{eq:nlp}. 

Augmented Lagrangian methods have been extensively used to solve nonlinear programming problems. These methods involve the iterative minimization of an augmented Lagrangian function with respect to the primal variables, followed by appropriate updates to the dual variables. Among the various methods, the most studied are the so--called Powell--Hestenes--Rockafellar (PHR) augmented Lagrangian methods \cite{hestenes1969multiplier,powell1969method,rockafellar1974augmented}, which are based on the function $\bar{L}_2 : \mathbbm{R}^n \times \mathbbm{R}^m \times \left( 0 , \infty \right) \to \mathbbm{R}$ such that 
\begin{equation} \label{phrAugLag}
\bar{L}_2 \left( x ; \lambda , r \right) = f \left( x \right) + \left\langle \lambda , h \left( x \right) \right\rangle + \frac{r}{2} \left\| h \left( x \right) \right\|^2 .
\end{equation}
This method has been studied extensively in the literature \cite{bertsekas1982constrained,conn1991globally,conn1996convergence,pennanen2002local,birgin2005numerical,andreani2008augmented,fernandez2012local,birgin2014practical}, and implemented successfully in software packages such as LANCELOT \cite{conn1992lancelot} and ALGENCAN \cite{andreani2007augmented}. We emphasize that at each iteration, the function to be minimized, $x \mapsto \bar{L}_2 \left( x ; \lambda , r \right)$, is continuously differentiable. Furthermore, convergence results can be obtained even if, instead of a global minimizer, an approximate stationary point is found at each iteration.

Another group of methods is based on the sharp augmented Lagrangian function $\bar{L}_1 : \mathbbm{R} \times \mathbbm{R}^m \times \left( 0 , \infty \right) \to \mathbbm{R}$ defined as follows:
\begin{equation} \label{sAugLag}
\bar{L}_1 \left( x ; \lambda , r \right) = f \left( x \right) + \left\langle \lambda , h \left( x \right) \right\rangle + r \left\| h \left( x \right) \right\| .
\end{equation}
Methods based on this function have been studied in \cite{gasimov2002augmented,burachik2006modified,jefferson2009thesis,burachik2010primal,kasimbeyli2009modified,bagirov2019sharp}. These studies propose a duality scheme that preserves the main idea of the modified subgradient algorithm, where at each iteration, a function of the dual variables is obtained by globally minimizing $x \mapsto \bar{L}_1 \left( x ; \lambda , r \right)$, and the dual variables are updated in the direction of a subgradient of the dual function. A practical drawback of these methods is that at each iteration, one must find a global minimizer, or a good approximation, of a nonlinear,  nonconvex, and nondifferentiable function. 

In this work, we propose a method based on the sharp augmented Lagrangian with a primal approach similar to that used in PHR augmented Lagrangian methods. To perform the primal minimization, we shall use a suitable smoothing technique. Among all possible approaches to smoothing out the kinks of the sharp augmented Lagrangian function, we choose the one introduced in \cite{fernandez2022augmented}. Specifically, it was shown that
\[
\bar{L}_1 \left( x ; \lambda , r \right) = \inf_{t > 0} \left\{ \textstyle \bar{L}_2 \left( x ; \lambda , \frac{r}{t} \right) + \frac{r}{2} t \right\}.
\]
This relation establishes the sharp augmented Lagrangian as a scalarization of the penalty parameter in the PHR augmented Lagrangian \eqref{phrAugLag}. By adopting this approach, we aim to inherit some of the desirable properties of the PHR augmented Lagrangian.

For parameters $\left( \lambda , r \right) \in  \mathbbm{R}^m \times \left( 0 , \infty \right)$, we define the extended--real--valued smoothing function as follows:
\[
\tilde{L}_{\lambda , r} \left( x , t \right) = \left\{ \begin{array}{ll}
f \left( x \right) + \left\langle \lambda , h \left( x \right) \right\rangle + \frac{r}{2 t} \left\| h \left( x \right) \right\|^2 + \frac{r}{2} t , & \quad t > 0 , \\[2mm]
f \left( x \right) , & \quad t = 0 , \, h \left( x \right) = 0 , \\[2mm]
\infty , & \quad \textrm{otherwise}.
\end{array} \right.
\]
It can be observed that $\tilde{L}_{\lambda , r}$ is a lower semicontinuous function on $\mathbbm{R}^n \times \mathbbm{R}$ and is continuously differentiable on $\mathbbm{R}^n  \times \left( 0 , \infty \right)$. Clearly, for this function we have $\bar{L}_1 \left( x ; \lambda , r \right) = \min_{t \geq 0} \left\{ \tilde{L}_{\lambda , r} \left( x , t \right) \right\}$.

The goal is to modify the sharp augmented Lagrangian method by replacing the minimization of the function $x \mapsto \bar{L}_1 \left( x ; \lambda , r \right)$ with the minimization of the function $\left( x , t \right) \mapsto \tilde{L}_{\lambda , r} \left( x , t \right)$. The nondifferentiability of $\bar{L}_1$ at $x$ where $h \left( x \right) = 0$ is addressed by introducing a singularity in $\tilde{L}_{\lambda , r}$ at $t = 0$. We will demonstrate that employing this smoothing technique does not compromise the classical duality scheme, which requires exact minimizers at each step. Additionally, we will present a primal approach where stationary points at each step are acceptable.

The rest of the paper is organized as follows. In Section \ref{sec:exact}, we examine the classical duality approach based on exact minimization for the smoothing function. Section \ref{sec:inexact} introduces the primal approach using stationary points, detailing two algorithms: one employing standard descent and the other using coordinate descent. Section \ref{boundingpenaltyparameter} studies the boundedness of the penaly parameter. Section \ref{sec:numerics} provides a comparison between the PHR augmented Lagrangian method and our two primal algorithms. Section \ref{sec:conclusions} is dedicated to conclusiones, and finally, an Appendix is included, describing a set of test problems chosen by the authors.

We conclude this section by defining our notation. We use $\langle \cdot , \cdot \rangle$ to denote the Euclidean inner product and $\| \cdot \|$ to represent the associated norm.

\section{Exact algorithm} \label{sec:exact}
We will prove that finding the global minimizer of $\tilde{L}_{\lambda , r}$ instead of the global minimizer of $\bar{L}_1 \left( \cdot ; \lambda , r \right)$ allows us to recover the main results from \cite[Chapter 2]{jefferson2009thesis}. To achieve this, we propose a method using a dual approach based on a modified subgradient algorithm. To this end, we refer to \eqref{eq:nlp} as the primal problem and define its associated augmented dual problem:
\[
\mathop{\textrm{maximize}}_{\left( \lambda , r \right) \in \mathbbm{R}^m \times \left( 0 , \infty \right)} \tilde{q} \left( \lambda , r \right) ,
\]
where
\[
\tilde{q} \left( \lambda , r \right) = \inf_{\left( x , t \right) \in \mathbbm{R}^n \times \mathbbm{R}} \tilde{L}_{\lambda , r} \left( x , t \right).
\]
Let us denote the set of minimizers of $\tilde{L}_{\lambda , r}$ by $A$, that is,
\[
A \left( \lambda , r \right) = \left\{ \left( x , t \right) \in \mathbbm{R}^n \times \mathbbm{R} \ \middle| \ \tilde{L}_{\lambda , r} \left( x , t \right) = \tilde{q} \left( \lambda , r \right) \right\}.
\]

For the exact case, we will use Algorithm \ref{SLModified}, which follows the classical dual approach based on the modified subgradient algorithm.

\

\begin{algorithm}[H]
\caption{Exact sharp Lagrangian (modified subgradient alg.)} \label{SLModified}
\begin{algorithmic}
\vspace{2mm}
\State \textbf{Step 0:} {\it Initialization} 
   
\vspace{1mm}
\noindent Choose $\left( \lambda^0 , r_0 \right) \in \mathbbm{R}^m \times \left( 0 , \infty \right)$ and a sequence of exogenous parameters $\left\{ \alpha_k \right\} \subset \left( 0 , \infty \right)$.

\noindent Set $k:=0$.

\vspace{2mm}
\State \textbf{Step 1:} {\it Solving the $k$--th subproblem} 

\vspace{1mm}
\begin{enumerate}[label=(\alph*)]
\item Find $\left( x^{k + 1} , t_{k + 1} \right) \in A \left( \lambda^k , r_k \right)$.
\item If $t_{k + 1} = 0$, STOP.
\item If $t_{k + 1} \neq 0$, go to {\bf Step 2}.
\end{enumerate}

\vspace{2mm}
\State \textbf{Step 2:} {\it Updating dual variables}

\vspace{1mm}
\noindent Set
\begin{align*}
\lambda^{k + 1} & = \lambda^k + \frac{r_k}{t_{k + 1}} h \left( x^{k + 1} \right) , \\
r_{k + 1} & = 2 r_k .
\end{align*}

\noindent Set $k := k + 1$ and go to {\bf Step 1}.
\end{algorithmic}
\end{algorithm}
\vspace{2mm}

In the next result we show some properties of the elements within the solution set $A$. One of these properties reveals a relationship between the minimizers of $\tilde{L}_{\lambda , r}$ and the subgradients of $- \tilde{q}$. To establish this, note first that $- \tilde{q}$ is a convex function, as it is the supremum of affine functions. Therefore, its subdifferential is given by
\begin{align*}
\partial \left( - \tilde{q} \right) \left( \hat{\lambda} , \hat{r} \right) = \Big\{ \left( \xi , \sigma \right) \mid - \tilde{q} \left( \lambda , r \right) & \geq - \tilde{q} \left( \hat{\lambda} , \hat{r} \right) \\
& + \left. \left\langle \left( \xi , \sigma \right) , \left( \lambda , r \right) - \left( \hat{\lambda} , \hat{r} \right) \right\rangle , \ \forall \left( \lambda , r \right) \right\} .
\end{align*}

\begin{lemma} \label{lem:Aproper}
Let $\left( \hat{x} , \hat{t} \right) \in A \left( \hat{\lambda} , \hat{r} \right)$, then it holds that
\begin{enumerate}[label=(\alph*)]
\item \label{t=normh(x)} $\hat{t} = \left\| h \left( \hat{x} \right) \right\|$, and
\item \label{subgrad_q} $- \left( h \left( \hat{x} \right) , \hat{t} \right) \in \partial \left( - \tilde{q} \right) \left( \hat{\lambda} , \hat{r} \right)$.
\end{enumerate}
\end{lemma}
\begin{proof}
\ref{t=normh(x)} If $\hat{t} > 0$, since $\tilde{L}_{\hat{\lambda} , \hat{r}}$ is continuously differentiable on $\mathbbm{R}^n \times \left( 0 , \infty \right)$, then
\[
0 = \frac{\partial \tilde{L}_{\hat{\lambda} , \hat{r}}}{\partial t} \left( \hat{x} , \hat{t} \right) = \frac{\hat{r}}{2} \left( 1 - \frac{\left\| h \left( \hat{x} \right) 
\right\|^2}{\hat{t}^2} \right).
\]
Thus, $\hat{t} = \left\| h \left( \hat{x} \right) \right\|$. In the case where $\hat{t} = 0$, it cannot occur that $h \left( \hat{x} \right) \neq 0$, because in that situation $\tilde{q} \left( \hat{\lambda} , \hat{r} \right) = \tilde{L}_{\hat{\lambda} , \hat{r}} \left( \hat{x} , \hat{t} \right) = \infty$. However, we know that $\tilde{q} \left( \hat{\lambda} , \hat{r} \right) \leq \tilde{L}_{\hat{\lambda} , \hat{r}} \left( x' , t' \right) < \infty$ for any $\left( x' , t' \right)$ with $t' > 0$. 
            
\noindent \ref{subgrad_q} Take $\left( \lambda , r \right) \in \mathbbm{R}^m \times \left( 0 , \infty \right)$. If $\hat{t} > 0$, then we have:
\begin{align*}
\tilde{q} \left( \lambda , r \right) \leq{} & \tilde{L}_{\lambda , r} \left( \hat{x} , \hat{t} \right) = f \left( \hat{x} \right) + \left\langle \lambda , h \left( \hat{x} \right) \right\rangle + \frac{r}{2 \hat{t}} \left\| h \left( \hat{x} \right) \right\|^2 + \frac{r}{2} \hat{t} \\[2mm]
={} & f \left( \hat{x} \right) + \left\langle \hat{\lambda} , h \left( \hat{x} \right) \right\rangle + \frac{\hat{r}}{2 \hat{t}} \left\| h \left( \hat{x} \right) \right\|^2 + \frac{\hat{r}}{2} \hat{t} + \left\langle h \left( \hat{x} \right) , \lambda -\hat{\lambda} \right\rangle \\[2mm]
+{} & \frac{\left( r - \hat{r} \right)}{2 \hat{t}} \left\| h \left( \hat{x} \right) \right\|^2 + \left( r - \hat{r} \right) \frac{\hat{t}}{2} \\[2mm]
={} & \tilde{q} \left( \hat{\lambda} , \hat{r} \right) + \left\langle \left( h \left( \hat{x} \right) , \hat{t} \right) , \left( \lambda , r \right) -\left( \hat{\lambda} , \hat{r} \right) \right\rangle ,
\end{align*}
where in the last equation we use $\left\| h \left( \hat{x} \right) \right\| = \hat{t}$ (item \ref{t=normh(x)}). For the case when $\hat{t} = 0$, by item \ref{t=normh(x)} we have $h \left( \hat{x} \right) = 0$. Thus:
\begin{align*}
\tilde{q} \left( \lambda , r \right) \leq{} & \tilde{L}_{\lambda , r} \left( \hat{x} , \hat{t} \right) = f \left( \hat{x} \right) = \tilde{L}_{\hat{\lambda} , \hat{r}} \left( \hat{x} , \hat{t} \right) = \tilde{q} \left( \hat{\lambda} , \hat{r} \right) \\[2mm]
={} & \tilde{q} \left( \hat{\lambda} , \hat{r} \right) + \left\langle \left( h \left( \hat{x} \right) , \hat{t} \right) , \left( \lambda , r \right) -\left( \hat{\lambda} , \hat{r} \right) \right\rangle .
\end{align*}
Hence, $- \left( h \left( \hat{x} \right) , \hat{t} \right) \in \partial \left( -\tilde{q} \right) \left( \hat{\lambda} , \hat{r} \right)$. 

\qed
\end{proof}

Now we will study the sequences generated by Algorithm \ref{SLModified}. First, note that if the smoothing parameter is zero, the primal point is feasible, since by Lemma \ref{lem:Aproper}\ref{t=normh(x)}, we have $\left\| h \left( x^k \right) \right\| = t_k = 0$. Also, the dual updates follow a subgradient direction of $- \tilde{q}$, i.e., a proximal point iteration of a linearization of $- \tilde{q}$. Note that the problem
\[
\begin{array}{ll} \displaystyle \mathop{\textrm{minimize}}_{(\lambda,r)} & - \tilde{q} \left( \lambda^k , r_k \right) - \left\langle \left( h \left( x^{k + 1} \right) , t_{k + 1} \right) , \left( \lambda , r \right) - \left( \lambda^k , r_k \right) \right\rangle \\[2mm] & + \displaystyle \frac{t_{k + 1}}{2 r_k} \left\| \left( \lambda , r \right) - \left( \lambda^k , r_k \right) \right\|^2 , \end{array}
\]
has the unique solution
\[
\left( \lambda^{k + 1} , r_{k + 1} \right) = \left( \lambda^k , r_k \right) + \frac{r_k}{t_{k + 1}} \left( h \left( x^{k + 1} \right) , t_{k + 1} \right) .
\]

Clearly, Algorithm \ref{SLModified} will generate a sequence if at Step 1 a global minimizer of $\tilde{L}_{\lambda^k , r_k}$ is found. The existence of such a minimizer can be ensured under various sets of assumptions. Since our result is independent of such assumptions, we will simply assume that $A \left( \lambda^k , r_k \right) \neq \emptyset$.

The next result states that if the algorithm stops, we obtain both a primal and a dual solution, otherwise the generated dual variables will approach the maximum of $\tilde{q}$. 

\begin{theorem} \label{thm:teo1} 
Suppose that $A \left( \lambda , r \right) \neq \emptyset$ for every $\left( \lambda , r \right) \in \mathbbm{R}^m \times \left( 0 , \infty \right)$. Then, the following holds:
\begin{enumerate}[label=(\alph*)]
\item \label{it:SLModStop} If Algorithm \ref{SLModified} stops at the $k$--th iteration, then $x^{k + 1}$ is an optimal primal solution, and $\left( \lambda^k , r_k \right)$ is an optimal dual solution. Moreover, the optimal values of the primal and augmented dual problems are the same.
\item \label{it:SLModNoStop} If $\left( \lambda^k , r_k \right)$ is not a dual solution, then $\tilde{q} \left( \lambda^{k + 1} , r_{k + 1} \right) \geq \tilde{q} \left( \lambda^k , r_k \right)$.
\end{enumerate}
\end{theorem}
\begin{proof} 
\ref{it:SLModStop} If the algorithm terminates at the $k$--th iteration, then $t_{k + 1} = 0$, and therefore $h \left( x^{k + 1} \right) = 0$ (see Lemma \ref{lem:Aproper}\ref{t=normh(x)}). Thus, $\left( 0 , 0 \right) \in \partial \left( - \tilde{q} \right) \left( \lambda^k , r_k \right)$. By the convexity of $- \tilde{q}$, for any $\left( \lambda , r \right) \in \mathbbm{R}^m \times \left( 0 , \infty \right)$ it holds that 
\[
\tilde{q} \left( \lambda , r \right) \leq  \tilde{q} \left( \lambda^k , r_k \right) ,
\]
which implies that $\left( \lambda^k , r_k \right)$ is an optimal dual solution. On the other hand, for any $x$ with $h \left( x \right) = 0$ we have
\[
f \left( x^{k + 1} \right) = \tilde{L}_{\lambda^k , r_k} \left( x^{k + 1} , t_{k + 1} \right) \leq \tilde{L}_{\lambda^k , r_k} \left( x , 0 \right) = f \left( x \right) ,
\]
which means $x^{k + 1}$ is an optimal primal solution. Moreover, there is no duality gap, since
\[
\tilde{q} \left( \lambda^k , r_k \right) = \tilde{L}_{\lambda^k , r_k} \left( x^{k + 1} , t_{k + 1} \right) = f \left( x^{k + 1} \right).
\]

\noindent \ref{it:SLModNoStop} If $\left( \lambda^k , r_k \right)$ is not a dual solution, Algorithm \ref{SLModified} will generate $\left( x^{k + 2} , t_{k + 2} \right) \in A \left( \lambda^{k + 1} , r_{k + 1} \right)$. First, consider the case when $t_{k + 2} > 0$. Given that $t_{k + 2} = \left\| h \left( x^{k + 2} \right) \right\|$ and using the update formula for $\lambda^{k + 1}$ and $r_{k + 1}$ we have:
\begin{align*}
\tilde{q} \left( \lambda^{k + 1} , r_{k + 1} \right) ={} & f \left( x^{k + 2} \right) + \left\langle \lambda^{k + 1} , h \left( x^{k + 2} \right) \right\rangle + \frac{r_{k + 1}}{2 t_{k + 2}} \left\| h \left( x^{k + 2} \right) \right\|^2 \\[2mm]
& + \frac{r_{k + 1}}{2} t_{k + 2} \\[2mm]
={} & f \left( x^{k + 2} \right) + \left\langle \lambda^k , h \left( x^{k + 2} \right) \right\rangle + \frac{r_k}{2 t_{k + 2}} \left\| h \left( x^{k + 2} \right) \right\|^2 + \frac{r_k}{2} t_{k + 2} \\[2mm]
& + \frac{r_k}{t_{k + 1}} \left\langle h \left( x^{k + 1} \right) , h \left( x^{k + 2} \right) \right\rangle + r_k t_{k + 2} \\[2mm]
\geq{} & \tilde{q} \left( \lambda^k , r_k \right) + r_k t_{k + 2} - \frac{r_k}{t_{k + 1}} \left\| h \left( x^{k + 1} \right) \right\| \left\| h \left( x^{k + 2} \right) \right\| \\[2mm]
={} & \tilde{q} \left( \lambda^k , r_k \right) .
\end{align*}

Now, consider the case when $t_{k + 2} = 0$. In this situation, the algorithm stops, which, by item \ref{it:SLModStop}, implies that $\left( \lambda^{k + 1} , r_{k + 1} \right)$ is an optimal dual solution. Since $\left( \lambda^k , r_k \right)$ is not a dual solution, it must be that $\tilde{q} \left( \lambda^{k + 1} , r_{k + 1} \right) > \tilde{q} \left( \lambda^k , r_k \right)$.

\qed
\end{proof}

Boundedness of the sequences generated by an algorithm is crucial for analyzing convergence. The following lemma describes what happens when the sequence of multipliers or the sequence of penalty parameters is bounded.

\begin{lemma} \label{acot_var_dual}
Suppose that $A \left( \lambda , r \right) \neq \emptyset$ for every $\left( \lambda , r \right) \in \mathbbm{R}^m \times \left( 0 , \infty \right)$. Consider that Algorithm \ref{SLModified} generates an infinite sequence $\left\{ \left( \lambda^k , r_k \right) \right\}$.
\begin{enumerate}[label=(\alph*)]
\item If $\left\{ r_k \right\}$ is bounded, then $\left\{ \lambda^k \right\}$ is bounded.
\item If $\left\{ \lambda^k \right\}$ is bounded and the set of optimal dual solutions is nonempty, then $\left\{ r^k \right\}$ is bounded.
\end{enumerate}
\end{lemma}
\begin{proof}
If an infinite sequence is generated, then $t_k > 0$ for all $k$. Then $t_k = \left\| h \left( x^k \right) \right\|$ and the first part follows from the following  expressions:
\begin{gather*}
\left\| \lambda_{k + 1} - \lambda_0 \right\| \leq \sum_{j = 0}^k \left\| \lambda_{j + 1} -\lambda_j \right\| = \sum_{j = 0}^k \frac{r_j}{t_{j + 1}} \left\| h \left( x^{j + 1} \right) \right\| = \sum_{j = 0}^k r_j , \\[2mm]
r_{k + 1} - r_0 = \sum_{j = 0}^k \left( r_{j + 1} - r_j \right) = \sum_{j = 0}^k r_j .
\end{gather*}
To prove the second statement, suppose that $\left\{ \lambda^k \right\}$ is bounded and take a dual solution $\left( \bar{\lambda} , \bar{r} \right)$. Since $- \left( h \left( x^{k + 1} \right) , t_{k + 1} \right) \in \partial \left( - \tilde{q} \right) \left( \lambda^k , r_k \right)$, we have that
\[
\tilde{q} \left( \bar{\lambda} , \bar{r} \right) \leq \tilde{q} \left( \lambda^k , r_k \right) + \left\langle h \left( x^{k + 1} \right) , \bar{\lambda} - \lambda^k \right\rangle + t_{k + 1} \left( \bar{r} - r_k \right) .
\]
Therefore,
\[
r_k \leq \frac{\tilde{q} \left( \lambda^k , r_k \right) - \tilde{q} \left( \bar{\lambda} , \bar{r} \right)}{t_{k + 1}} + \frac{\left\langle h \left( x^{k + 1} \right) , \bar{\lambda} - \lambda^k \right\rangle}{t_{k + 1}} + \bar{r} \leq \left\| \bar{\lambda} - \lambda^k \right\| + \bar{r} , 
\]
where we use the fact that $\left( \bar{\lambda} , \bar{r} \right)$ is a dual solution, the Cauchy--Schwarz inequality, and that $t_{k + 1} = \left\| h \left( x^{k + 1} \right) \right\|$. Consequently, if $\left\{ \lambda_k \right\}$ is bounded we conclude that $\left\{ r_k \right\}$ must also be bounded.

\qed
\end{proof}

The next lemma provides a sufficient condition for the boundedness of the sequence of penalty parameters.

\begin{lemma} \label{acot_rk}
Under the assumptions of Lemma \ref{acot_var_dual}, if the set of optimal dual solutions is nonempty, then the sequence $\left\{ r_k \right\}$ is bounded.
\end{lemma}
\begin{proof}
Let $\left( \bar{\lambda} , \bar{r} \right)$ be a dual solution. To arrive at a contradiction, assume that $\left\{ r_k \right\}$ is unbounded. Since $\left\{r_k  \right\}$ is increasing, there exists $k_0$ such that $2 \bar{r} < r_k$ for all $k \geq k_0$. For $k \geq k_0$, using the concavity of $\tilde{q}$, we have:
\begin{align*}
\left\| \bar{\lambda} - \lambda^{k + 1} \right\|^2 & = \left\| \bar{\lambda} - \left( \lambda^k + \frac{r_k}{t_{k + 1}} h \left( x^{k + 1} \right) \right) \right\|^2 \\[2mm]
& = \left\| \bar{\lambda} - \lambda^k \right\|^2 + \frac{r_k^2}{t_{k + 1}^2} \left\| h \left( x^{k + 1} \right) \right\|^2 - 2 \frac{r_k}{t_{k + 1}} \left\langle \bar{\lambda} - \lambda^k , h \left( x^{k + 1} \right) \right\rangle \\[2mm]
& \leq \left\| \bar{\lambda} - \lambda^k \right\|^2 + r_k^2 + 2 \frac{r_k}{t_{k + 1}} \left( \tilde{q} \left( \lambda^k , r_k \right) - \tilde{q} \left( \bar{\lambda} , \bar{r} \right) + t_{k + 1} \left( \bar{r} - r_k \right) \right) \\[2mm]
& \leq \left\| \bar{\lambda} - \lambda^k \right\|^2 + r_k \left( 2 \bar{r} - r_k \right) \\[2mm]
& \leq \left\| \bar{\lambda} - \lambda^k \right\|^2 ,
\end{align*}
where we used the fact that $\tilde{q} \left( \lambda^k , r_k \right) \leq \tilde{q} \left( \bar{\lambda} , \bar{r} \right)$. Thus, $\left\{ \lambda^k \right\}$ is bounded, and by Lemma \ref{acot_var_dual}, $\left\{ r_k \right\}$ must also be bounded, leading to a contradiction.

\qed
\end{proof}

The following theorem guarantees the finite termination of Algorithm \ref{SLModified}, provided that dual solutions exist.

\begin{theorem}
Suppose that $A \left( \lambda , r \right) \neq \emptyset$ for every $\left( \lambda , r \right) \in \mathbbm{R}^m \times \left( 0 , \infty \right)$ and that the set of dual solutions is nonempty. Then, there exists $\bar{k} > 0$ such that Algorithm \ref{SLModified} stops at the $\bar{k}$--th iteration. In particular, $x^{\bar{k} + 1}$ is a primal solution, and $\left( \lambda^{\bar{k}} , r_{\bar{k}} \right)$ is a dual solution.
\end{theorem}
\begin{proof}
To derive a contradiction, suppose that Algorithm \ref{SLModified} does not have finite termination. Then, for all $k$, we obtain 
\[
r_k - r_0 = \sum_{j = 0}^{k - 1} \left( r_{j + 1} - r_j \right) = \sum_{j = 0}^{k - 1} r_j \geq \sum_{j = 0}^{k - 1} r_0 = k r_0 ,
\]
which contradicts the boundedness of $\left\{ r_k \right\}$ as established by Lemma \ref{acot_rk}. Therefore, there exists $\bar{k}$ such that $t_{\bar{k} + 1} = 0$, and, by Theorem \ref{thm:teo1}\ref{it:SLModStop}, $x^{\bar{k} + 1}$ is a primal solution and $\left( \lambda^{\bar{k}} , r_{\bar{k}} \right)$ is a dual solution.

\qed
\end{proof}

It is important to emphasize that in the case of no duality gap, i.e., when $\tilde{q} \left( \lambda , r \right) = f \left( x \right)$ for feasible dual and primal points $\left( \lambda , r \right)$ and $x$, respectively, the vector $\lambda$ is considered as a Lagrange multiplier in an extended sense. The structure of our augmented dual function enables us to show the existence of a Lagrange multiplier that satisfies \eqref{eq:lagrange}. 

\begin{proposition}
Suppose that $A \left( \lambda , r \right) \neq \emptyset$ for every $\left( \lambda , r \right) \in \mathbbm{R}^m \times \left( 0 , \infty \right)$. If Algorithm \ref{SLModified} terminates at the $k$--th iteration, then $x^{k + 1}$ is a stationary point of \eqref{eq:nlp}. 
\end{proposition}
\begin{proof}
From Theorem \ref{thm:teo1}\ref{it:SLModStop}, we know that $h \left( x^{k + 1} \right) = 0$ and $\tilde{q} \left( \lambda^k , r_k \right) = f \left( x^{k + 1} \right)$. Therefore,
\[
\bar{L}_1 \left( x^{k + 1} ; \lambda^k , r_k \right) = f \left( x^{k + 1} \right) = \tilde{q} \left( \lambda^k , r_k \right) \leq \tilde{L}_{\lambda^k , r_k} \left( x , t \right) ,
\]
for all $\left( x , t \right)$. Taking the minimum for $t \geq 0$ we obtain that $\bar{L}_1 \left( x^{k + 1} ; \lambda^k , r_k \right) \leq \bar{L}_1 \left( x ; \lambda^k , r_k \right)$. Hence, $x^{k + 1}$ is a minimizer of $\bar{L}_1$ and thus
\[
0 \in \partial \bar{L}_1 \left( x^{k + 1} ; \lambda^k , r_k \right) = \left\{ \nabla  f \left( x^{k + 1} \right) + \nabla h \left( x^{k + 1} \right) \left( \lambda^k + r_k \xi \right) \mid \left\| \xi \right\| \leq 1 \right\} .
\]
Therefore, $0 = \nabla f \left( x^{k + 1} \right) + \nabla h \left( x^{k + 1} \right) \left(  \lambda^k + r_k \xi^k \right)$ for some $\xi^k$ with $\left\| \xi^k \right\| \leq 1$. In summary, $x^{k + 1}$ is a stationary point with the associated Lagrange multiplier $\lambda^k + r_k \xi^k$.
    
\qed
\end{proof}

Note that Step 1a of Algorithm \ref{SLModified} requires to tackle two issues: 
\begin{itemize}
\item dealing with a nondifferentiable objective function when $t = 0$,
\item finding an exact global minimizer.
\end{itemize}

Regarding the nondifferentiability, it is possible that all solutions are of the form $\left( x , 0 \right)$ where $\tilde{L}_{\lambda , r}$ is not continuous and hence not differentiable. Additionally, since most optimization solvers assume some degree of smoothness for the objective function and constraints, numerical issues may arise. 

Furthermore, global optimization algorithms tend to be slow or unreliable for large-- or medium-- scale problems \cite{birgin2014practical}. Moreover, exact solutions are unattainable due to the limitations of finite arithmetic in computers.

In preliminary implementations of Algorithm \ref{SLModified}, we observed that Step 1a could not finish satisfactorily when $t_k$ was close to zero and the penalty parameter $r_k$ became excessively large.

In the next section we propose a strategy to address the nondifferentiability at $t = 0$. We also relax the requirement of finding an exact solution to the subproblem, allowing for an inexact stationary point.

\section{Inexact Methods}\label{sec:inexact}
In this section, we present a primal approach to the previous algorithm. The main advantage of this approach is that it allows us to replace exact minimizers with inexact stationary points. It is well known that for nonsmooth functions, such as $x \mapsto \left| x \right|$, there may be no points in a neighborhood of the minimizer where the derivative of the function is close to zero. This behavior can also be observed in $\tilde{L}_{\lambda , r}$ when applied to simple problems, as illustrated in the next example.

\begin{example} \label{ejem: L tilde no funciona}
Consider the following problem:
\begin{equation} \label{eq: L tilde no funciona}
\begin{array}{cl}
\textrm{minimize} & \frac{1}{2} x^2 \\[2mm]
\textrm{subject to} & x = 0 .
\end{array}
\end{equation}
The solution is $\bar{x} = 0$ with the associated Lagrange multiplier $\bar{\lambda} = 0$. For this problem, for each pair $\left( \lambda , r \right)$, we have
\[
\tilde{L}_{\lambda , r} \left( x , t \right) = \left\{ \begin{array}{ll}
\frac{1}{2} x^2 + \lambda x + \frac{r}{2 t} x^2 + \frac{r}{2} t , & \quad t > 0 , \\[2mm]
0 , & \quad x = t = 0 , \\[2mm]
\infty , & \quad \textrm{otherwise} .
\end{array} \right.
\]

Let $r \geq \left| \lambda \right| + c$ for some $c > 0$, and suppose there exists $\left( x , t \right)$ with $t > 0$ such that $\left\| \nabla \tilde{L}_{\lambda , r} \left( x , t \right) \right\| \leq \varepsilon$. Then, there exist $p$ and $q$ such that $p^2 + q^2 \leq 1$ and
\begin{align}
\varepsilon q & = \frac{\partial \tilde{L}_{\lambda , r}}{\partial x} \left( x , t \right) = x + \lambda + \frac{r}{t} x , \label{eps_est_x} \\[2mm]
\varepsilon p & = \frac{\partial \tilde{L}_{\lambda , r}}{\partial t} \left( x , t \right) = \frac{r}{2} \left( 1 - \frac{x^2}{t^2} \right) . \label{eps_est_t}
\end{align}
From \eqref{eps_est_x}, we have that $x = \left( \varepsilon q - \lambda \right) t / \left( t + r \right)$, and from \eqref{eps_est_t}, we get $0 = \left( 1 - 2 \varepsilon p / r \right)  t^2 - x^2$. Combining these two equations, we obtain
\begin{align*}
0 & = \left( 1 - \frac{2 \varepsilon p}{r} \right) \left( t + r \right)^2 t^2 - \left( \varepsilon q - \lambda \right)^2 t^2 \\[2mm]
& = \left[ \left( 1 - \frac{2 \varepsilon p}{r} \right) \left( t + r \right)^2 - \left( \varepsilon q - \lambda \right)^2 \right] t^2 .
\end{align*}
Since $t > 0$, for $\varepsilon$ small enough such that $1 - 2 \varepsilon p / r > 0$, we get 
\[
t = \frac{\left| \varepsilon q - \lambda \right|}{\sqrt{ 1 - 2 \varepsilon p / r}} - r 
\leq \frac{\varepsilon + \left| \lambda \right|}{\sqrt{1 - 2 \varepsilon p / r}} - r
\leq \frac{\varepsilon + r - c }{\sqrt{1 - 2 \varepsilon p / r}} - r .
\]
This leads to a contradiction, as the rightmost term is negative for sufficiently small $\varepsilon$. Hence, there is no $\left( x , t \right)$ with $t > 0$ such that $\left\| \nabla \tilde{L}_{\lambda , r} \left( x , t \right) \right\| \leq \varepsilon$ for sufficiently small $\varepsilon$.
\end{example}

To ensure the existence of inexact stationary points for a continuously differentiable function, we introduce a barrier function associated with the condition $t \geq 0$. For computational simplicity, we employ the inverse barrier function. The function to be minimized is defined as:
\begin{equation}
\hat{L}^s_{\lambda , r} \left( x , t \right) = \tilde{L}_{\lambda , r} \left( x , t \right) + \frac{r}{2 t} s^2 , \label{L tilde + barrera}
\end{equation}
where $s > 0$ is the barrier parameter. This function is clearly lower semicontinuous and continuously differentiable for $t > 0$. The existence of inexact stationary points is shown in the following proposition.

\begin{proposition} \label{existencia (x(s),t(s))}
Assume that $f$ is bounded from below and $h$ is bounded. Then, for any $s > 0$ and $\tilde{\varepsilon} > 0$, there exists a pair $\left( x \left( s \right) , t \left( s \right) \right)$ such that $\left\| \nabla \hat{L}^s_{\lambda , r} \left( x \left( s \right) , t \left( s \right) \right) \right\| \leq \tilde{\varepsilon}$.
\end{proposition}
\begin{proof}
Let $f \left( x \right) \geq \kappa_f$ and $\left\| h \left( x \right) \right\| \leq \kappa_h$ hold for all $x$. Then, the following holds:
\begin{itemize}
\item If $r \geq \left\| \lambda \right\|$, we have $\left( r - \left\| \lambda \right\| \right) \left\| h \left( x \right) \right\| \geq 0$.
\item If $r < \left\| \lambda \right\|$, then $\left( r - \left\| \lambda \right\| \right) \left\| h \left( x \right) \right\| \geq \left( r - \left\| \lambda \right\| \right) \kappa_h$.
\end{itemize}
Thus, for all $s > 0$ and $\left( x , t \right) \in \mathbbm{R}^n \times \left( 0 , \infty \right)$ we have
\begin{align*}
\hat{L}^s_{\lambda , r} \left( x , t \right) & \geq \tilde{L}_{\lambda , r} \left( x , t \right) \geq \bar{L}_1 \left( x ; \lambda , r \right) \geq f \left( x \right) + \left( r - \left\| \lambda \right\| \right) \left\| h \left( x \right) \right\| \\[2mm]
& \geq \kappa_f + \min \left\{ 0 , \left( r - \left\| \lambda \right\| \right) \kappa_h \right\} .
\end{align*}
Hence, $\hat{L}^s_{\lambda , r}$ is bounded from below. Now, take $\eta > 0$ and $\left( \tilde{x} , \tilde{t} \right)$ such that 
$\hat{L}^s_{\lambda , r} \left( \tilde{x} , \tilde{t} \right) < \inf_{\left( x , t \right)} \hat{L}^s_{\lambda , r} \left( x , t \right) + \eta$. Then, by \cite[Proposition 10.44]{rockafellar2009variational}, there exists $\left( x \left( s \right) , t \left( s \right) \right)$ such that $\hat{L}^s_{\lambda , r} \left( x \left( s \right) , t \left( s \right) \right) \leq \hat{L}^s_{\lambda , r} \left( \tilde{x} , \tilde{t} \right)$ and a vector $v \in \partial \hat{L}^s_{\lambda , r} \left( x \left( s \right) , t \left( s \right) \right)$ with $\left\| v \right\| \leq \tilde{\varepsilon}$. Notice that by the definition of $\hat{L}^s_{\lambda , r}$, it follows that $t \left( s \right) > 0$, implying that $\hat{L}^s_{\lambda , r}$ is continuously differentiable at $\left( x \left( s \right) , t \left( s \right) \right)$. Consequently, we have
\[
\left\| \nabla \hat{L}^s_{\lambda , r} \left( x \left( s \right) , t \left( s \right) \right) \right\| = \left\| v \right\| \leq \tilde{\varepsilon} .
\]
\qed
\end{proof}

\begin{remark} \label{rmk:cutoff}
It is not hard to see that the objective and constraint functions can be redefined to meet the conditions of Proposition \ref{existencia (x(s),t(s))}. For instance, given some $M > 0$, we could define $\tilde{f} \left( x \right) = e^{f \left( x \right)}$ and $\tilde{h}_i \left( x \right) = \max \left\{ - M + \tanh \left( h_i \left( x \right) + M \right) , \min \left\{ h_i \left( x \right) , M + \tanh \left( h_i \left( x \right) - M \right) \right\} \right\}$. Consequently, $\tilde{f}$ is bounded from below, $\tilde{h}$ is bounded, and both functions are twice continuously differentiable.    
\end{remark}

Recall that $\nabla \hat{L}^s_{\lambda , r} \left( x , t \right) = \left( \nabla_x  
\hat{L}^s_{\lambda , r} \left( x , t \right) , \frac{\partial \hat{L}^s_{\lambda , r}}{\partial t} \left( x , t \right) \right)$ where
\begin{align}
\nabla_x \hat{L}^s_{\lambda , r} \left( x , t \right) & = \nabla_x \tilde{L}_{\lambda , r} \left( x , t \right) = \nabla f \left( x \right) + \nabla h \left( x \right) \left( \lambda + \frac{r}{t} h \left( x \right) \right) , \label{dhatL_dx} \\[2mm]
\frac{\partial \hat{L}^s_{\lambda , r}}{\partial t} \left( x , t \right) & = \frac{r}{2} \left( 1 - \frac{\left\| h \left( x \right) \right\|^2 + s^2}{t^2} \right) . \label{dhatL_dt}  
\end{align}
These expressions will be useful in the forthcoming calculations.

\subsection{Fixing the smoothing parameter}
According to Proposition \ref{existencia (x(s),t(s))}, there exists $\left( x \left( s \right) , t \left( s \right) \right)$, an inexact stationary point  of $\hat{L}^s_{\lambda , r}$. To compute such a point, note that, as indicated by \eqref{dhatL_dt}, the stationary point of the function $t \mapsto \hat{L}^s_{\lambda , r} \left( x , t \right)$, and hence a global minimizer due to convexity, must satisfy
\[
t = \sqrt{\left\| h \left( x \right) \right\|^2 + s^2} .
\]
We propose a coordinate descent--like strategy to find this point. In the $k$--th iteration, we define $t_{k + 1}$ as the minimizer of $t \mapsto \hat{L}^{s_k}_{\lambda^k , r_k} \left( x^k , t \right)$, and $x^{k + 1}$ as an inexact stationary point of the problem minimizing $x \mapsto \hat{L}^{s_k}_{\lambda^k , r_k} \left( x , t_{k + 1} \right)$. The existence of such $x^{k + 1}$ is guaranteed by similar arguments to those in Proposition \ref{existencia (x(s),t(s))}. Based on this approach, we define Algorithm \ref{SL Inex con Punto Interior}.

\begin{algorithm}
\caption{Inexact sharp Lagrangian (fixed smoothing parameter)} \label{SL Inex con Punto Interior}
\begin{algorithmic}
\vspace{2mm}
\State \textbf{Step 0:} {\it Initialization} 
     
\vspace{1mm}
\noindent Take $\lambda_{\min} < \lambda_{\max}$, $0 < \tau < 1$, $\gamma > 1$, and $tol > 0$. Choose $x^0 \in \mathbbm{R}^n$, $t_0 > 0$, $\bar{\lambda}^0 \in \left[ \lambda_{\min} , \lambda_{\max} \right]^m$, and $r_0 > 0$. Consider sequences of exogenous parameters $\varepsilon_k \searrow 0$ and $\left\{ s_k \right\} \subset \left( 0 , \infty \right)$ that are bounded. 

\noindent Set $k := 0$.

\vspace{2mm}
\State \textbf{Step 1:} {\it Checking stopping criterion} 

\vspace{1mm}
\[
\text{If } \sqrt{ \left\| \nabla_x \tilde{L}_{\bar{\lambda}^k , r_k} \left( x^k , t_k \right) \right\|^2 + \left\| h \left( x^k \right) \right\|^2} \leq tol , \text{ then STOP} .
\]
    
\vspace{2mm}
\State \textbf{Step 2:} {\it Updating primal variables}

\vspace{1mm}
\noindent Define
\[
t_{k + 1} = \sqrt{\left\| h \left( x^k \right) \right\|^2 + s_k^2} .
\]
     
\noindent Find $x^{k + 1} \in \mathbbm{R}^n$ such that
\[
\left\| \nabla_x \tilde{L}_{\bar{\lambda}^k , r_k} \left( x^{k + 1} , t_{k + 1} \right)  \right\| \leq \varepsilon_k . 
\]

\vspace{2mm}
\State \textbf{Step 3:} {\it Updating Lagrangian multipliers}

\vspace{1mm}
\noindent Define 
\[
\lambda^{k + 1} = \bar{\lambda}^k + r_k \frac{h \left( x^{k + 1} \right)}{t_{k + 1}} . 
\]

\vspace{2mm}
\State \textbf{Step 4:} {\it Updating penalty parameter}

\vspace{1mm}
\noindent If $\left\| h \left( x^{k + 1} \right) \right\| \leq \tau \left\| h \left( x^k \right) \right\|$, choose $r_{k + 1} = r_k$. Otherwise, define $r_{k + 1} = \gamma r_k$.

\vspace{2mm} 
\State \textbf{Step 5:} {\it Projecting Lagrangian multipliers}
     
\vspace{1mm}
\noindent Compute $\bar{\lambda}^{k + 1} \in \left[ \lambda_{\min} , \lambda_{\max} \right]^m$.
 
\noindent Set $k := k + 1$ and go to Step 1.     
\end{algorithmic}
\end{algorithm}

The next proposition shows that every limit point of the sequence generated by Algorithm \ref{SL Inex con Punto Interior} is either a stationary point of problem \eqref{eq:nlp} or a stationary point of the problem that minimizes infeasibility.

\begin{proposition}
Let $\bar{x} \in \mathbbm{R}^n$ be a limit point of the sequence $\left\{ x^k \right\}$ generated by Algorithm \ref{SL Inex con Punto Interior}. Then:
\begin{enumerate}[label=(\alph*)]
\item If $\left\{ r_k \right\}$ is bounded, then $\bar{x}$ is a stationary point of problem \eqref{eq:nlp}.
\item If $\left\{ r_k \right\}$ is unbounded and $\left\{ h \left( x^k \right) \right\}$ is bounded, then $\bar{x}$ is a stationary point of the unconstrained problem of minimizing $\left\| h \left( x \right) \right\|^2$.
%\item[ii)] If ${r_k}$ is unbounded, $\bar{x}$ is a stationary point of the unconstrained problem of minimizing $\|h(x)\|^2$.
\end{enumerate}
\end{proposition}
\begin{proof}
Let $\bar{x}$ be a limit point of the sequence $\left\{ x^k \right\}$ generated by Algorithm \ref{SL Inex con Punto Interior}. Then there exists an infinite subset $\mathcal{K} \subset \mathbbm{N}$ such that $\lim_{\underset{k \in \mathcal{K}}{k \to \infty}} x^{k + 1} = \bar{x}$. With this in mind, we consider two cases.

\

\noindent (a) If $\left\{ r_k \right\}$ is bounded, then by Step 4 of Algorithm \ref{SL Inex con Punto Interior}, there exists $k_0 \in \mathbbm{N}$ such that $r_k = r_{k_0}$ and $\left\| h \left( x^{k + 1} \right) \right\| \leq \tau \left\| h \left( x^k \right) \right\|$ for all $k \geq k_0$. Therefore, for all $k \geq k_0$,
\[
\frac{\left\| h \left( x^{k + 1} \right) \right\|}{t_{k + 1}} 
 = \frac{\left\| h \left( x^{k + 1} \right) \right\|}{\sqrt{\left\| h \left( x^k \right) \right\|^2 + s_k^2}} \leq \frac{\tau \left\| h \left( x^k \right) \right\|}{\sqrt{\left\| h \left( x^k \right) \right\|^2 + s_k^2}} \leq \tau .
\]

Taking the limit over a suitable subsequence, from the inequality in Step 2 of Algorithm \ref{SL Inex con Punto Interior}, we have
\[
\nabla f \left( \bar{x} \right) + \nabla h \left( \bar{x} \right) \left( \bar{\lambda} +  r_{k_0} \xi \right) = 0 .
\]
for some $\xi$, which is a limit point of $\left\{ h \left( x^{k + 1} \right) / t_{k + 1} \right\}_{k \in \mathcal{K}}$, and for some $\bar{\lambda}$, a limit point of $\left\{ \bar{\lambda}^k \right\}_{k \in \mathcal{K}}$. 
%Notar que $\|\xi\|\leq 1$.
    
Furthermore, by Step 4 of Algorithm \ref{SL Inex con Punto Interior}, we have that $h \left( \bar{x} \right) = 0$, indicating that $\bar{x}$ is a feasible point for problem \eqref{eq:nlp}. Therefore, $\bar{x}$ is a stationary point for this problem.

\

\noindent (b) If $\left\{ r_k \right\}$ is unbounded and $\left\{ h \left( x^k \right) \right\}$ is bounded, then $1 / r_k \to 0$ and $t_{k + 1}$ is bounded (since $\left\{ s_k \right\}$ is bounded).
    
On the other hand, if we multiply both sides of the inequality in Step 2 of Algorithm \ref{SL Inex con Punto Interior} by $t_{k + 1} / r_k$, we obtain
\[
\left\| \frac{t_{k + 1}}{r_k} \nabla f \left( x^{k + 1} \right) + \nabla h \left( x^{k + 1} \right) \left[ \frac{t_{k + 1}}{r_k} \bar{\lambda}^k + h \left( x^{k + 1} \right) \right] \right\| \leq \frac{t_{k + 1}}{r_k} \varepsilon_k . 
\]
Taking limits as $k \in \mathcal{K}$, we get 
\[
\nabla h \left( \bar{x} \right) h \left( \bar{x} \right) = 0 .
\] 
Consequently, $\bar{x}$ is a stationary point for the problem of minimizing $\left\| h \left( x \right) \right\|^2$.
    
\qed
\end{proof}

The boundedness condition of $\left\{ h \left( x^k \right) \right\}$ in item (b) can be ensured by using a cutoff function, as described in Remark \ref{rmk:cutoff}.

\subsection{Varying the smoothing parameter}
In the previous section, we decoupled the computation of $\left( x^{k + 1} , t_{k + 1} \right)$. As a result, this pair is not necessarily an inexact stationary point of $\hat{L}^s_{\bar{\lambda}^k , r_k}$. In this section, however, we will treat this point as an inexact stationary point of $\hat{L}^s_{\bar{\lambda}^k , r_k}$. By doing so, we will obtain results similar to those in the aforementioned section.

\begin{algorithm}
\caption{Inexact sharp Lagrangian (varying the smoothing parameter)} \label{SL Inex dos variables}
\begin{algorithmic}
\vspace{2mm}
\State \textbf{Step 0:} {\it Initialization} 
     
\vspace{1mm}
\noindent Take $\lambda_{\min} < \lambda_{\max}$, $0 < \tau < 1$, $\gamma > 1$, $tol > 0$. Choose $x^0 \in \mathbbm{R}^n$, $t_0 > 0$, $\bar{\lambda}^0 \in \left[ \lambda_{\min} , \lambda_{\max} \right]^m$, $r_0 > 0$. Take sequences of exogenous parameters $\varepsilon_k \searrow 0$ and $\left\{ s_k \right\} \subset \left( 0 , \infty \right)$ bounded. 

\noindent Set $k := 0$.
     
\vspace{2mm}
\State \textbf{Step 1:} {\it Checking stopping criterion} 

\vspace{1mm}
\noindent 
\[
\text{If } \sqrt{\left\| \nabla_x \tilde{L}_{\bar{\lambda}^k , r_k} \left( x^k , t_k \right) \right\|^2 + \left\| h \left( x^k \right) \right\|^2} \leq tol , \text{ then STOP} .
\]

\vspace{2mm}
\State \textbf{Step 2:} {\it Updating primal variables}

\vspace{1mm}
\noindent Find $\left( x^{k + 1} , t_{k + 1} \right) \in \mathbbm{R}^n \times \mathbbm{R}_{++}$ such that
\[
\left\| \nabla \hat{L}^{s_k}_{\bar{\lambda}^k , r_k} \left( x^{k + 1} , t_{k + 1} \right) \right\| \leq \varepsilon_k , 
\]

\vspace{2mm}
\State \textbf{Step 3:} {\it Updating Lagrangian multipliers}

\vspace{1mm}
\noindent Define 
\[
\lambda^{k + 1} = \bar{\lambda}^k + r_k \frac{h \left( x^{k + 1} \right)}{t_{k + 1}} . 
\]

\vspace{2mm}
\State \textbf{Step 4:} {\it Updating penalty parameter}

\vspace{1mm}
\noindent If $\left\| h \left( x^{k + 1} \right) \right\| \leq \tau \left\| h \left( x^k \right) \right\|$, choose $r_{k + 1} = r_k$. Otherwise, define $r_{k + 1} = \gamma r_k$.

\vspace{2mm} 
\State \textbf{Step 5:} {\it Projecting Lagrangian multipliers}
     
\vspace{1mm}
\noindent Compute $\bar{\lambda}^{k + 1} \in \left[ \lambda_{\min} , \lambda_{\max} \right]^m$.
    
\noindent Set $k := k + 1$ and go to Step 1.     

\end{algorithmic}
\end{algorithm}

\begin{remark} \label{h_k/t_k y t_k acotadas}
From Step 2 of Algorithm \ref{SL Inex dos variables} it can be deduced that the sequence $\left\{ \left\| h \left( x^k \right) \right\| / t_k \right\}$ is bounded and also that $\left\{ t_k \right\}$ is bounded if the sequence $\left\{ h \left( x^k \right) \right\}$ is bounded. Indeed, from \eqref{dhatL_dt}, we have
\[
\left| \frac{\partial \hat{L}^{s_k}_{\lambda^k , r_k}}{\partial t} \left( x^{k + 1} , t_{k + 1} \right) \right| \leq \varepsilon_k ,
\]
which holds if and only if
\[
1 - \frac{2 \varepsilon_k}{r_k} \leq \frac{\left\| h \left( x^{k + 1} \right) \right\|^2 + s_k^2}{t_{k + 1}^2} \leq 1 + \frac{2 \varepsilon_k}{r_k} .
\]
Thus,
\[
\frac{\left\| h \left( x^{k + 1} \right) \right\|}{t_{k + 1}} \leq \sqrt{1 + \frac{2 \varepsilon_k}{r_k}} , \qquad t_{k+1} \leq \sqrt{\frac{\left\| h \left( x^{k + 1} \right) \right\|^2 + s_k^2}{1 - \frac{2 \varepsilon_k}{r_k}}} .
\]
\end{remark}

The following proposition shows that every limit point of the sequence generated by Algorithm \ref{SL Inex dos variables} is either a stationary point of problem \eqref{eq:nlp} or a stationary point of the problem that minimizes infeasibility.

\begin{proposition}
Let $\bar{x} \in \mathbbm{R}^n$ be a limit point of the sequence $\left\{ x^k \right\}$ generated by Algorithm \ref{SL Inex dos variables}. Then:
\begin{enumerate}[label=(\alph*)]
\item If $\left\{ r_k \right\}$ is bounded, then $\bar{x}$ is a stationary point of problem \eqref{eq:nlp}.
\item If $\left\{ r_k \right\}$ is unbounded, then $\bar{x}$ is a stationary point of the unconstrained problem of minimizing $\left\| h \left( x \right) \right\|^2$.
\end{enumerate}
\end{proposition}
\begin{proof}
Let $\bar{x}$ be a limit point of the sequence $\left\{ x^k \right\}$ generated by Algorithm \ref{SL Inex dos variables}. Then there exists an infinite subset $\mathcal{K} \subset \mathbbm{N}$ such that $\lim_{\underset{k \in \mathcal{K}}{k \to \infty}} x^{k + 1} = \bar{x}$. With this in mind, we consider two cases. 

\

\noindent (a) Suppose $\left\{ r_k \right\}$ is bounded. Then, by Step 4 of Algorithm \ref{SL Inex dos variables}, there exists $k_0 \in \mathbbm{N}$ such that $r_k = r_{k_0}$ and $\left\| h \left( x^{k + 1} \right) \right\| \leq \tau \left\| h \left( x^k \right) \right\|$ for all $k \geq k_0$. Consequently, $h \left( \bar{x} \right) = 0$, which means that $\bar{x}$ is a feasible point for problem \eqref{eq:nlp}.
    
On the other hand, by Remark \ref{h_k/t_k y t_k acotadas}, the sequence $\left\{ h \left( x^{k + 1} \right) / t_{k + 1} \right\}_{k \in \mathcal{K}}$ is bounded. Taking the limit over a suitable subsequence, from the inequality in Step 2 of Algorithm \ref{SL Inex dos variables}, we obtain 
\[
\nabla f \left( \bar{x} \right) + \nabla h \left( \bar{x} \right) \left( \bar{\lambda} + r_{k_0} \xi \right) = 0 .
\]
for some $\xi$, which is a limit point of $\left\{ h \left( x^{k + 1} \right) / t_{k + 1} \right\}_{k \in \mathcal{K}}$, and for some $\bar{\lambda}$, a limit point of $\left\{ \bar{\lambda}^k \right\}_{k \in \mathcal{K}}$. 
    
Therefore, $\bar{x}$ is a stationary point for the problem of minimizing $f \left( x \right)$ subject to $h \left( x \right) = 0$.

\

\noindent (b) Suppose $\left\{ r_k \right\}$ is unbounded. Since $\left\{ h \left( x^{k + 1} \right) \right\}_{k \in \mathcal{K}}$ is bounded (because $\left\{ x^{k + 1} \right\}_{k \in \mathcal{K}}$ is convergent and $h$ is a continuous function), it follows from Remark \ref{h_k/t_k y t_k acotadas} that $\left\{ t_{k + 1} \right\}_{k \in \mathcal{K}}$ is also bounded.
    
On the other hand, if we multiply both sides of the inequality in Step 2 of Algorithm \ref{SL Inex dos variables} by $t_{k + 1} / r_k$ we obtain
\[
\left\| \frac{t_{k + 1}}{r_k} \nabla f \left( x^{k + 1} \right) + \nabla h \left( x^{k + 1} \right) \left[ \frac{t_{k + 1}}{r_k} \bar{\lambda}^k + h \left( x^{k + 1} \right) \right] \right\| \leq \frac{t_{k + 1}}{r_k} \varepsilon_k . 
\]
Taking limits as $k \in \mathcal{K}$, we get
\[
\nabla h \left( \bar{x} \right) h \left( \bar{x} \right) = 0 .
\] 
Consequently, $\bar{x}$ is a stationary point for the problem of minimizing $\left\| h \left( x \right) \right\|^2$.

\qed
\end{proof}

\section{Boundedness of the penalty parameters} \label{boundingpenaltyparameter}
We have observed that, in both Algorithm \ref{SL Inex con Punto Interior} and Algorithm \ref{SL Inex dos variables}, if the sequence of penalty parameters is bounded, the limit points of the sequence generated by these algorithms are stationary points of the primal problem. This raises the question of when it is possible to guarantee such a bound, that is, under what conditions the sequence of penalty parameters remains bounded. This issue is of interest not only from a theoretical point of view but also from a computational one, since when the penalty parameters are excessively large, the subproblems tend to be ill--conditioned, making their solution more difficult. More generally, the study of conditions that ensure the bounding of penalty parameters is fundamental in approaches based on augmented Lagrangians (see \cite{andreani2007augmented,fernandez2012local,birgin2014practical}).

Henceforth, we will assume that the following hypotheses hold.
\begin{enumerate}[label=(\textit{H\arabic*}), ref=\arabic*, leftmargin=*]
\item \label{hip:1} The sequences $\left\{ x^k \right\}$, $\left\{ \lambda^k \right\}$ generated by the application of Algorithm \ref{SL Inex con Punto Interior} or Algorithm \ref{SL Inex dos variables} satisfy the following conditions: $\lim_{k \to \infty} x^k = \bar{x}$ and $\bar{\lambda}^{k + 1} = P_{\left[ \lambda_{\min} , \lambda_{\max} \right]^m} \left( \lambda^{k + 1} \right)$, orthogonal projection onto $\left[ \lambda_{\min} , \lambda_{\max} \right]^m$.
\item \label{hip:2} The point $\bar{x}$ is feasible (that is, $h \left( \bar{x} \right) = 0$).
\item \label{hip:3} At the point $\bar{x}$, the LICQ condition is satisfied, meaning that the gradients $\nabla h_i \left( \bar{x} \right)$, $i = 1 , \ldots , m$ are linearly independent.
\item \label{hip:4} The second--order sufficient optimality condition is satisfied at $\bar{x}$, with Lagrange multiplier $\bar{\lambda} \in \mathbbm{R}^m$.
\item \label{hip:5} It is satisfied that $\bar{\lambda} \in \left[ \lambda_{\min} , \lambda_{\max} \right]^m$.
\end{enumerate}

\begin{remark} \label{El multiplicador es único}
Note that, due to Hypothesis (\textit{H\ref{hip:3}}), the Lagrange multiplier mentioned in Hypothesis (\textit{H\ref{hip:4}}) is unique.
\end{remark}

\begin{proposition}
Suppose that Hypotheses (\textit{H\ref{hip:1}}), (\textit{H\ref{hip:2}}), (\textit{H\ref{hip:3}}) and (\textit{H\ref{hip:5}}) hold. Then $\lim_{k \to \infty} \bar{\lambda}^k = \bar{\lambda}$.
\end{proposition}
\begin{proof}
First, let us see that the sequence $\left\{ \lambda^k \right\}$ is bounded. Suppose, in fact, that is not; then there exists a subsequence $\left\{ \lambda^k \right\}_{k \in \mathcal{K}}$ of $\left\{ \lambda^k \right\}$, where $\mathcal{K} \subset \mathbbm{N}$ is infinite, such that $\lim_{\underset{k \in \mathcal{K}}{k \to \infty}} \left\| \lambda^k \right\| = \infty$. Furthermore, in both Algorithm \ref{SL Inex con Punto Interior} and Algorithm \ref{SL Inex dos variables}, from Step 2 and the definition of $\lambda^{k + 1}$ it follows that, for each $k \in \mathbbm{N}_0$,
\[
\left\| \nabla f \left( x^{k + 1} \right) + \nabla h \left( x^{k + 1} \right) \lambda^{k + 1} \right\| \leq \varepsilon_k .
\]
Or, equivalently,
\begin{equation} \label{gradL en (xnuevo, lambdanuevo) menor a epsilon}
\left\| \nabla f \left( x^k \right) + \nabla h \left( x^k \right) \lambda^k \right\| \leq \varepsilon_{k - 1} , \quad \text{for all } k \in \mathbbm{N}.
\end{equation}
Consequently, for each $k\in\mathcal{K}$, it follows that
\[
\left\| \frac{\nabla f \left( x^k \right)}{\left\| \lambda^k \right\|} + \nabla h \left( x^k \right) \frac{\lambda^k}{\left\| \lambda^k \right\|} \right\| \leq \frac{\varepsilon_{k - 1}}{\left\| \lambda^k \right\|} .
\]
Thus, by taking the limit over $k$ in $\mathcal{K}$ and using Hypothesis (\textit{H\ref{hip:1}}), we obtain
\[
\nabla h \left( \bar{x} \right) \tilde{\lambda} = 0 ,
\]
where $\tilde{\lambda}$ is a limit point of $\left\{ \frac{\lambda^k}{\left\| \lambda^k \right\|} \right\}_{k \in \mathcal{K}}$. Since $\left\| \tilde{\lambda} \right\| = 1$, this contradicts Hypothesis (\textit{H\ref{hip:3}}). Therefore, $\left\{ \lambda^k \right\}$ must be bounded.

Since $\left\{ \lambda^k \right\}$ is a bounded sequence, to show that $\lim_{k \to \infty} \lambda^k = \bar{\lambda}$, it is sufficient to prove that $\bar{\lambda}$ is its unique limit point. To this end, let $\hat{\lambda}$ be a limit point of $\left\{ \lambda^k \right\}$. We will demonstrate that $\hat{\lambda} = \bar{\lambda}$. Since $\hat{\lambda}$ is a limit point of $\left\{ \lambda^k \right\}$, there exists an infinite set $\mathcal{K} \subset \mathbbm{N}$, $\mathcal{K}$ such that $\lim_{\underset{k \in \mathcal{K}}{k \to \infty}} \lambda^k = \hat{\lambda}$. Thus, by taking the limit of \eqref{gradL en (xnuevo, lambdanuevo) menor a epsilon} over $k$ in $\mathcal{K}$, we obtain that
\[
\nabla f \left( \bar{x} \right) + \nabla h \left( \bar{x} \right) \hat{\lambda} = 0 .
\]
Therefore, by Remark \ref{El multiplicador es único}, we have $\hat{\lambda} = \bar{\lambda}$. Thus, $\bar{\lambda}$ is the unique limit point of the sequence $\left\{ \lambda^k \right\}$.

Finally, by the continuity of the projection, it follows that $\lim_{k \to \infty} \bar{\lambda}^k = \bar{\lambda}$.

\qed
\end{proof}

\begin{theorem}
Let us consider either Algorithm \ref{SL Inex con Punto Interior} or Algorithm \ref{SL Inex dos variables}, and assume that Hypotheses (\textit{H\ref{hip:1}}) through (H\textit{\ref{hip:5}}) are satisfied. Additionally, suppose that $s_k = O \left( \left\| h \left( x^k \right) \right\| \right)$ and $\varepsilon_k = o \left( \left\| h \left( x^k \right) \right\| \right)$. Then, the sequence of penalty parameters $\left\{ r_k \right\}$ is bounded.
\end{theorem}
\begin{proof}
First, note that $\frac{r_k}{t_{k + 1}} \to \infty$. Indeed, from Hipotheses (\textit{H\ref{hip:1}}) and (\textit{H\ref{hip:2}}), and the continuity of $h$, we have $\lim_{k \to \infty} h \left( x^k \right) = 0$. Since $s_k = O \left( \left\| h \left( x^k \right) \right\| \right)$, it follows that $t_{k + 1} \to 0$. Given that $\left\{ r_k \right\}$ is non--decreasing, the result follows.

On the other hand, recall that $\tilde{L}_{\lambda , r} \left( x , t \right) = \bar{L} \left( x ; \lambda , \frac{r}{t} \right) + \frac{r}{2} t$ when $t > 0$. Considering this, along with the previous paragraph and using Hypotheses (\textit{H\ref{hip:1}}), (\textit{H\ref{hip:3}}) and (\textit{H\ref{hip:4}}), we obtain, according to \cite[Prop. 4.2]{fernandez2012local}, that there exists $M > 0$ such that, for sufficiently large $k$,
\begin{equation} \label{acotación del paper de Damián y Solodov}
\left\| x^{k + 1} - \bar{x} \right\| + \left\| \lambda^{k + 1} - \bar{\lambda} \right\| \leq M \left( \varepsilon_k + \frac{t_{k + 1}}{r_k} \left\| \bar{\lambda}^k - \bar{\lambda} \right\| \right) .
\end{equation}

Furthermore, since $h \left( \bar{x} \right) = 0$, and due to the continuity of the first derivatives of $h$, there exists $L > 0$ such that, for all $k \in \mathbbm{N}_0$,
\[
\left\| h \left( x^{k + 1} \right) \right\| \leq L \left\| x^{k + 1} - \bar{x} \right\| .
\]
Therefore, by \eqref{acotación del paper de Damián y Solodov}, it follows that for sufficiently large $k$,
\begin{equation} \label{norm de h(x(k+1)) debe ser o pequeño de h(xk)}
\left\| h \left( x^{k + 1} \right) \right\| \leq LM \left( \varepsilon_k + \frac{t_{k + 1}}{r_k} \left\| \bar{\lambda}^k - \bar{\lambda} \right\| \right) .  
\end{equation}

Now, it is necessary to distinguish between two cases, depending on the algorithm used to obtain $t_{k + 1}$: 
\begin{itemize}
\item In the case where $t_{k + 1}$ is obtained from the application of Algorithm \ref{SL Inex con Punto Interior}, according to Step 2, we have $t_{k + 1} = \sqrt{\left\| h \left( x^k \right) \right\|^2 + s_k^2}$. Since $s_k = O \left( \left\| h \left( x^k \right) \right\| \right)$, it follows that $t_{k + 1} = O \left( \left\| h \left( x^k \right) \right\| \right)$. Consequently, given that $\lim_{k \to \infty} \bar{\lambda}^k = \bar{\lambda}$, we obtain that $\frac{t_{k + 1}}{r_k} \left\| \bar{\lambda}^k - \bar{\lambda} \right\| = o \left( \left\| h \left( x^k \right) \right\| \right)$. Therefore, since $\varepsilon_k = o \left( \left\| h \left( x^k \right) \right\| \right)$, from \eqref{norm de h(x(k+1)) debe ser o pequeño de h(xk)}, it follows that $\left\| h \left( x^{k + 1} \right) \right\| = o \left( \left\| h \left( x^k \right) \right\| \right)$.
\item In the case where $t_{k + 1}$ arises from the application of Algorithm \ref{SL Inex dos variables}, by Remark \ref{h_k/t_k y t_k acotadas}, we have that
\[
t_{k + 1} \leq \sqrt{\frac{\left\| h \left( x^{k + 1} \right) \right\|^2 + s_k^2}{1 - \frac{2 \varepsilon_k}{r_k}}} .
\]

Then, since $1 - \frac{2 \varepsilon_k}{r_k} \to 1$ and $s_k = O \left( \left\| h \left( x^k \right) \right\| \right)$, it follows that $t_{k + 1} = O \left( \left\| h \left( x^{k + 1} \right) \right\| \right) + O \left( \left\| h \left( x^k \right) \right\| \right)$. Therefore, applying a similar reasoning to the previous case, we obtain $\left\| h \left( x^{k + 1} \right) \right\| = o \left( \left\| h \left( x^{k + 1} \right) \right\| \right) + o \left( \left\| h \left( x^k \right) \right\| \right)$. This implies, as easily verified, that $\left\| h \left( x^{k + 1} \right) \right\| = o \left( \left\| h \left( x^k \right) \right\| \right)$.
\end{itemize}

In summary, we have shown that, regardless of the algorithm used to obtain $t_{k + 1}$, it holds that $\left\| h \left( x^{k + 1} \right) \right\| = o \left( \left\| h \left( x^k \right) \right\| \right)$. Therefore, $\left\| h \left( x^{k + 1} \right) \right\| \leq \tau \left\| h \left( x^k \right) \right\|$ for sufficiently large $k$. Consequently, the sequence $\left\{ r_k \right\}$ must be bounded.

\qed
\end{proof}

\section{Numerical results} \label{sec:numerics}
In this section, we will analyze a series of problems from the base \cite{hock1980test}, selecting only those problems with equality constraints: problems 6--9, 26--28, 39--40, 42, 47--52, 56, 61, 77--79. Additionally, problems formulated by the authors (501--514) have been included, which feature either one or two constraints, a domain dimension of at most 3, and involve simple functions (typically linear or quadratic). Some of these problems have finite feasible sets, while others have infinite but easily described sets (such as a spheres or a circles). Problem 508 uses the well--known Rosenbrock function as the objective function \cite{smith2020}.

For the practical implementations of ALGENCAN, we used the default parameters. For both Algorithm \ref{SL Inex con Punto Interior} and Algorithm \ref{SL Inex dos variables}, the parameters were set as follows: $\lambda_{\min} = - 10^{20}$, $\lambda_{\max} = 10^{20}$, $\tau = 0.9$, $\gamma = 10$, $tol = 10^{- 8}$, $x_0$ as specified for each problem, $t_0 = 1$, $\bar{\lambda}^0 = 0$, $r_0 = 10$.

All experiments were conducted on a PC running Linux, Core(TM) i7-10510U CPU \@ 1.80GHz with 32 GB of RAM. The algorithms were implemented in GNU Fortran 95, using compiler version 4:11.2.0.

Step 2 of Algorithms \ref{SL Inex con Punto Interior} and \ref{SL Inex dos variables} was solved by GENCAN, which is part of the well--known software ALGENCAN \cite{birgin2014practical,birgin2000nonmonotone,andreani2008augmented} (version 3.1.1). GENCAN \cite{birgin2002large,andreani2010second,andretta2005practical,birgin2001box,birgin2008structured} is an active set method with projected gradients designed for bound--constrained minimization.

The test problems were solved with ALGENCAN, Algorithm \ref{SL Inex con Punto Interior}, and Algorithm \ref{SL Inex dos variables}. The exit flags of ALGENCAN are stored in the variable {\it Inform}, and can take the following values:
\begin{itemize}
\item {\it 0:} Solution was found (according to the stopping criteria of ALGENCAN).
\item {\it 1:} The penalty parameter is too large. The problem may be infeasible or badly scaled. Further analysis is required.
\item {\it 2:} Maximum of iterations reached. Feasibility--complementarity and optimality tolerances could not be achieved. Whether the final iterate is a solution or not requires further analysis.
\item {\it 3:} It seems that a stationary--of--the--infeasibility probably infeasible point was found. Whether the final iterate is a solution or not requires further analysis.
\end{itemize}

The exit flags for the other two algorithms are also stored in the variable {\it Inform}, with the following values:
\begin{itemize}
\item {\it 0:} Solution was found (i.e., Step 1 was successfully completed).
\item {\it 1:} Maximum number of iterations reached.
\end{itemize}

Tables \ref{table:algencan}, \ref{table:alg2}, \ref{table:alg3} show the results using ALGENCAN, Algorithm \ref{SL Inex con Punto Interior}, and Algorithm \ref{SL Inex dos variables}, respectively. Each column of the tables represents, from left to right:
\begin{itemize}
\item Prob.: problem number to be solved.
\item It.: number of external iterations.
\item KKT norm: KKT norm of the last calculated point.
\item Int. It.: the sum of all internal iterations during the algorithm's execution.
\item Inform: exit flag, as explained above.
\item $x$: last point determined by the algorithm.
\item $f \left( x \right)$: value of the objective function at the last approximation.
\item $\lambda$: last Lagrange multiplier obtained.
\item Infeas.: norm of infeasibility at the final point.
\end{itemize}

It is important to highlight that the stopping criteria for ALGENCAN differ from those of Algorithms \ref{SL Inex con Punto Interior} and \ref{SL Inex dos variables}. Specifically, the latter algorithms stop when the KKT norm is less than a predefined tolerance, whereas ALGENCAN employs a different criterion. This discrepancy explains why there are problems where ALGENCAN reports a zero output but does not meet the KKT norm tolerance required by Algorithms \ref{SL Inex con Punto Interior} and \ref{SL Inex dos variables}. 

Considering these stopping criteria, we observe that ALGENCAN was unable to solve problems 56, 79, 511 and 512, whereas Algorithms \ref{SL Inex con Punto Interior} and \ref{SL Inex dos variables} could not solve problems 26, 56, 79 and 511.

Some preliminary conclusions can be drawn from inspecting the Tables:
\begin{itemize}
\item The three algorithms solve all problems except four. ALGENCAN could not solve problems 56, 79, 511, 512, while Algorithms 2 and 3 could not solve problems 26, 56, 79 and 511. This means that ALGENCAN was the winner on problem 26, and the other two algorithms were the winners on 512.
\item Considering the number of internal iterations, Algorithm 2 is better than Algorithm 3.
\item Considering the number of internal iterations, Algorithm 2 is a good competitor to ALGENCAN (in fact, it performs fewer internal iterations in almost all problems).
\item All three algorithms obtain reasonable good primal approximations but fail to solve the Lagrangian multipliers.
\end{itemize}

The numerical experiments shown in the result tables suggest that our algorithms are competitive when compared to the well--known ALGENCAN software. This encourages us to investigate deeper into this novel strategy in order to be able to solve a wider range of test problems.

%TABLA ALGENCAN
{\scriptsize
\begin{landscape} 
 
\begin{longtable}{>
{\centering\arraybackslash}m{1.0cm}|>{\centering\arraybackslash}m{1.0cm}|>{\centering\arraybackslash}m{2.0cm}|>{\centering\arraybackslash}m{1.0cm}|>{\centering\arraybackslash}m{1.2cm}|>{\centering\arraybackslash}m{2.0cm}|>{\centering\arraybackslash}m{2.0cm}|>{\centering\arraybackslash}m{2.0cm}|>{\centering\arraybackslash}m{2.0cm}}
Prob. & It. & KKT norm & Int. it. & Inform & $x$ & $f \left( x \right)$ & $\lambda$ & Infeas. \\[2mm] \hline
\endfirsthead
Prob. & It. & KKT norm & Int. it. & Inform & $x$ & $f \left( x \right)$ & $\lambda$ & Infeas. \\[2mm] \hline
\endhead
\hline
\multicolumn{9}{r}{\textit{Continued on the next page.}}
\endfoot
\endlastfoot
 
6  &   6 & 3.19E-10 &  23 & 0 &  1.0000E+00 &  7.6980E-20 &  3.1164E-11 &  5.2225E-12 \\ 
   &     &          &     &   &  1.0000E+00 &             &             &             \\[2mm] \hline 
7  &  15 & 4.00E-10 &  49 & 0 & -1.8908E-13 & -1.7321E+00 &  2.8868E-01 &  4.0008E-10 \\ 
   &     &          &     &   &  1.7321E+00 &             &             &             \\[2mm] \hline 
8  &   3 & 7.94E-14 &   9 & 0 &  4.6016E+00 & -1.0000E+00 & -1.6095E-15 &  7.3949E-14 \\ 
   &     &          &     &   &  1.9558E+00 &             & -3.0103E-15 &             \\[2mm] \hline 
9  &   4 & 3.44E-10 &   8 & 0 & -7.5000E+01 & -5.0000E-01 & -3.2725E-02 &  2.6745E-10 \\ 
   &     &          &     &   & -1.0000E+02 &             &             &             \\[2mm] \hline 
26 &   5 & 5.96E-09 &  37 & 0 &  1.0004E+00 &  3.6300E-13 & -2.4609E-10 &  1.4997E-10 \\ 
   &     &          &     &   &  1.0004E+00 &             &             &             \\ 
   &     &          &     &   &  9.9961E-01 &             &             &             \\[2mm] \hline 
27 &   8 & 8.43E-09 &  55 & 0 & -1.0000E+00 &  4.0000E-02 &  4.0000E-02 &  3.8751E-11 \\ 
   &     &          &     &   &  1.0000E+00 &             &             &             \\ 
   &     &          &     &   & -1.1131E-08 &             &             &             \\[2mm] \hline 
28 &   7 & 2.23E-09 &  33 & 0 &  5.0000E-01 &  1.5406E-18 & -1.2872E-09 &  1.9018E-11 \\ 
   &     &          &     &   & -5.0000E-01 &             &             &             \\ 
   &     &          &     &   &  5.0000E-01 &             &             &             \\[2mm] \hline 
39 &  20 & 1.97E-10 & 125 & 0 &  1.0000E+00 & -1.0000E+00 & -1.0000E+00 &  1.7795E-10 \\ 
   &     &          &     &   &  1.0000E+00 &             & -1.0000E+00 &             \\ 
   &     &          &     &   &  1.0262E-14 &             &             &             \\ 
   &     &          &     &   & -1.2755E-14 &             &             &             \\[2mm] \hline 
40 &   8 & 1.33E-09 &  54 & 0 &  7.9370E-01 & -2.5000E-01 &  5.0000E-01 &  5.1885E-10 \\ 
   &     &          &     &   &  7.0711E-01 &             & -4.7194E-01 &             \\ 
   &     &          &     &   &  5.2973E-01 &             &  3.5355E-01 &             \\ 
   &     &          &     &   &  8.4090E-01 &             &             &             \\[2mm] \hline 
42 &   7 & 4.05E-09 &  33 & 0 &  2.0000E+00 &  1.3858E+01 & -2.0000E+00 &  8.3106E-11 \\ 
   &     &          &     &   &  2.0000E+00 &             &  2.5355E+00 &             \\ 
   &     &          &     &   &  8.4853E-01 &             &             &             \\ 
   &     &          &     &   &  1.1314E+00 &             &             &             \\[2mm] \hline 
47 &   7 & 4.34E-09 &  68 & 0 &  1.0000E+00 &  8.7943E-19 & -7.8935E-11 &  9.1590E-11 \\ 
   &     &          &     &   &  1.0000E+00 &             &  1.7514E-09 &             \\ 
   &     &          &     &   &  1.0000E+00 &             & -8.9889E-10 &             \\ 
   &     &          &     &   &  1.0000E+00 &             &             &             \\ 
   &     &          &     &   &  1.0000E+00 &             &             &             \\[2mm] \hline 
48 &   6 & 1.82E-08 &  22 & 0 &  1.0000E+00 &  3.5793E-17 & -2.9494E-09 &  1.2964E-12 \\ 
   &     &          &     &   &  1.0000E+00 &             &  4.5016E-10 &             \\ 
   &     &          &     &   &  1.0000E+00 &             &             &             \\ 
   &     &          &     &   &  1.0000E+00 &             &             &             \\ 
   &     &          &     &   &  1.0000E+00 &             &             &             \\[2mm] \hline 
49 &   7 & 3.02E-07 &  51 & 0 &  1.0096E+00 &  5.4099E-10 &  5.1837E-08 &  3.3812E-11 \\ 
   &     &          &     &   &  1.0096E+00 &             & -2.9788E-08 &             \\ 
   &     &          &     &   &  1.0000E+00 &             &             &             \\ 
   &     &          &     &   &  9.9518E-01 &             &             &             \\ 
   &     &          &     &   &  1.0000E+00 &             &             &             \\[2mm] \hline 
50 &   5 & 2.82E-08 &  20 & 0 &  1.0000E+00 &  6.6509E-18 &  6.2155E-09 &  2.1283E-10 \\ 
   &     &          &     &   &  1.0000E+00 &             & -1.0453E-08 &             \\ 
   &     &          &     &   &  1.0000E+00 &             &  5.1301E-09 &             \\ 
   &     &          &     &   &  1.0000E+00 &             &             &             \\ 
   &     &          &     &   &  1.0000E+00 &             &             &             \\[2mm] \hline 
51 &   6 & 1.74E-09 &  19 & 0 &  1.0000E+00 &  2.8382E-20 &  6.3542E-10 &  7.2982E-11 \\ 
   &     &          &     &   &  1.0000E+00 &             &  9.5253E-11 &             \\ 
   &     &          &     &   &  1.0000E+00 &             & -1.4789E-09 &             \\ 
   &     &          &     &   &  1.0000E+00 &             &             &             \\ 
   &     &          &     &   &  1.0000E+00 &             &             &             \\[2mm] \hline 
52 &   8 & 2.46E-08 &  42 & 0 & -9.4556E-02 &  5.3266E+00 &  3.2779E+00 &  1.6924E-11 \\ 
   &     &          &     &   &  3.1519E-02 &             &  2.9054E+00 &             \\ 
   &     &          &     &   &  5.1576E-01 &             & -7.7479E+00 &             \\ 
   &     &          &     &   & -4.5272E-01 &             &             &             \\ 
   &     &          &     &   &  3.1519E-02 &             &             &             \\[2mm] \hline 
56 &  15 & 5.55E+37 &  34 & 1 & -3.7198E+16 &  1.3968E+47 & -9.5844E+36 &  6.0282E+16 \\ 
   &     &          &     &   &  3.4676E+16 &             &  8.9346E+36 &             \\ 
   &     &          &     &   &  1.0829E+14 &             &  2.7902E+34 &             \\ 
   &     &          &     &   & -1.9582E+18 &             &  9.2673E+35 &             \\ 
   &     &          &     &   & -1.9282E+18 &             &             &             \\ 
   &     &          &     &   & -1.9291E+18 &             &             &             \\ 
   &     &          &     &   &  2.5304E+18 &             &             &             \\[2mm] \hline 
61 &   6 & 2.79E-08 &  18 & 0 &  5.3268E+00 & -1.4365E+02 & -8.8768E-01 &  3.3420E-11 \\ 
   &     &          &     &   & -2.1190E+00 &             & -1.7378E+00 &             \\ 
   &     &          &     &   &  3.2105E+00 &             &             &             \\[2mm] \hline 
77 &  10 & 5.53E-09 &  54 & 0 &  1.1662E+00 &  2.4151E-01 & -8.5540E-02 &  4.5998E-11 \\ 
   &     &          &     &   &  1.1821E+00 &             & -3.1878E-02 &             \\ 
   &     &          &     &   &  1.3803E+00 &             &             &             \\ 
   &     &          &     &   &  1.5060E+00 &             &             &             \\ 
   &     &          &     &   &  6.1092E-01 &             &             &             \\[2mm] \hline 
78 &   8 & 8.65E-09 &  38 & 0 & -1.7171E+00 & -2.9197E+00 &  7.4445E-01 &  6.7312E-10 \\ 
   &     &          &     &   &  1.5957E+00 &             & -7.0358E-01 &             \\ 
   &     &          &     &   &  1.8272E+00 &             &  9.6806E-02 &             \\ 
   &     &          &     &   & -7.6364E-01 &             &             &             \\ 
   &     &          &     &   & -7.6364E-01 &             &             &             \\[2mm] \hline 
79 &  50 & 2.11E+13 & 396 & 2 &  3.0206E+04 &  2.4224E+14 & -1.3817E+09 &  5.1635E-05 \\ 
   &     &          &     &   & -3.5822E+02 &             & -2.9089E+11 &             \\ 
   &     &          &     &   & -5.4120E+01 &             &  2.9505E+08 &             \\ 
   &     &          &     &   &  3.2900E+03 &             &             &             \\ 
   &     &          &     &   &  6.6213E-05 &             &             &             \\[2mm] \hline 
501 &  26 & 6.28E-10 &  47 & 0 &  1.0000E+00 & -1.5000E+00 &  5.5556E-01 &  6.2821E-10 \\[2mm] \hline 
502 &   3 & 7.20E-11 &   3 & 0 &  7.1957E-11 &  2.5889E-21 & -7.1956E-11 &  7.1957E-11 \\[2mm] \hline 
503 &   3 & 4.73E-11 &   3 & 0 &  2.3672E-11 &  1.1207E-21 & -4.7344E-11 &  4.7344E-11 \\ 
    &     &          &     &   &  2.3672E-11 &             &             &             \\[2mm] \hline 
504 &  13 & 3.93E-10 &  10 & 0 & -1.0000E+00 &  0.0000E+00 & -8.7294E-11 &  3.9268E-10 \\[2mm] \hline 
505 &   8 & 1.78E-10 &  65 & 0 & -1.8724E-19 & -1.0000E+00 &  1.5000E+00 &  1.9016E-11 \\ 
    &     &          &     &   & -1.0000E+00 &             &             &             \\ 
    &     &          &     &   &  1.4775E-20 &             &             &             \\[2mm] \hline 
506 &  16 & 7.43E-10 &  41 & 0 & -7.0711E-01 & -1.4142E+00 &  7.0711E-01 &  7.4303E-10 \\ 
    &     &          &     &   & -7.0711E-01 &             &             &             \\[2mm] \hline 
507 &  14 & 2.92E-10 &  30 & 0 & -1.0000E+00 & -1.0000E+00 & -5.0000E-01 &  2.9178E-10 \\[2mm] \hline 
508 &   7 & 5.34E-03 &  21 & 0 &  1.0000E+00 &  8.7977E-08 & -6.5693E-03 &  3.6369E-12 \\ 
    &     &          &     &   &  1.0000E+00 &             &             &             \\[2mm] \hline 
509 &   5 & 1.61E-08 &  24 & 0 &  6.0000E+00 & -1.0800E+02 &  1.5000E+00 &  1.5200E-10 \\ 
    &     &          &     &   &  3.0000E+00 &             &             &             \\[2mm] \hline 
510 &   7 & 4.08E-09 &  35 & 0 & -5.3452E-01 & -3.7417E+00 &  1.8708E+00 &  2.3976E-11 \\ 
    &     &          &     &   & -8.0178E-01 &             &             &             \\ 
    &     &          &     &   & -2.6726E-01 &             &             &             \\[2mm] \hline 
511 &  50 & 1.61E-06 & 377 & 2 &  1.7203E-10 & -2.7453E-05 &  3.6425E+04 &  4.1482E-10 \\ 
    &     &          &     &   & -2.7453E-05 &             & -1.8212E+04 &             \\[2mm] \hline 
512 &   0 & 1.73E+00 &   0 & 3 &  0.0000E+00 &  0.0000E+00 &  0.0000E+00 &  1.0000E+00 \\ 
    &     &          &     &   &  0.0000E+00 &             &             &             \\[2mm] \hline 
513 &   2 & 1.66E-10 &   5 & 0 &  1.6601E-10 & -7.5943E-40 & -1.3280E-19 &  1.6601E-10 \\[2mm] \hline 
514 &   6 & 1.91E-10 &   6 & 0 &  1.0000E+00 &  5.0000E-01 & -1.0000E+00 &  1.9056E-10 \\     &     &          &     &   & -4.8412E-50 &             &             &             \\[2mm] \hline 
 
\caption{Problems solved by ALGENCAN.}
\label{table:algencan} \\ 
\end{longtable}
\end{landscape}
}

%TABLA ALGORITMO 2
{\scriptsize
\begin{landscape} 
 
\begin{longtable}{>
{\centering\arraybackslash}m{1.0cm}|>{\centering\arraybackslash}m{1.0cm}|>{\centering\arraybackslash}m{2.0cm}|>{\centering\arraybackslash}m{1.0cm}|>{\centering\arraybackslash}m{1.2cm}|>{\centering\arraybackslash}m{2.0cm}|>{\centering\arraybackslash}m{2.0cm}|>{\centering\arraybackslash}m{2.0cm}|>{\centering\arraybackslash}m{2.0cm}}
Prob. & It. & KKT norm & Int. it. & Inform & $x$ & $f \left( x \right)$ & $\lambda$ & Infeas. \\[2mm] \hline
\endfirsthead
Prob. & It. & KKT norm & Int. it. & Inform & $x$ & $f \left( x \right)$ & $\lambda$ & Infeas. \\[2mm] \hline
\endhead
\hline
\multicolumn{9}{r}{\textit{Continued on the next page.}}
\endfoot
\endlastfoot

6  &   6 & 6.59E-09 &  48 & 0 &  1.0000E+00 &  1.7749E-30 & -6.2058E-15 &  1.5465E-12 \\ 
   &     &          &     &   &  1.0000E+00 &             &             &             \\[2mm] \hline 
7  &   6 & 1.34E-09 &  18 & 0 & -2.8635E-11 & -1.7321E+00 &  2.8868E-01 &  2.0242E-12 \\ 
   &     &          &     &   &  1.7321E+00 &             &             &             \\[2mm] \hline 
8  &   4 & 2.95E-12 &   9 & 0 &  4.6016E+00 & -1.0000E+00 & -2.8633E-13 &  0.0000E+00 \\ 
   &     &          &     &   &  1.9558E+00 &             & -2.4159E-14 &             \\[2mm] \hline 
9  &   3 & 5.51E-10 &   4 & 0 & -1.8750E+03 & -5.0000E-01 & -3.2725E-02 &  0.0000E+00 \\ 
   &     &          &     &   & -2.5000E+03 &             &             &             \\[2mm] \hline 
26 & 100 & 9.65E-02 & 337 & 1 &  9.9490E-01 &  1.0590E-08 & -1.9556E-02 &  0.0000E+00 \\ 
   &     &          &     &   &  9.9490E-01 &             &             &             \\ 
   &     &          &     &   &  1.0050E+00 &             &             &             \\[2mm] \hline 
27 &   7 & 1.65E-11 &  27 & 0 & -1.0000E+00 &  4.0000E-02 &  4.0000E-02 &  7.9936E-15 \\ 
   &     &          &     &   &  1.0000E+00 &             &             &             \\ 
   &     &          &     &   & -1.2091E-15 &             &             &             \\[2mm] \hline 
28 &   5 & 4.44E-12 &   6 & 0 &  5.0000E-01 &  2.3179E-24 &  1.1578E-12 &  6.6613E-16 \\ 
   &     &          &     &   & -5.0000E-01 &             &             &             \\ 
   &     &          &     &   &  5.0000E-01 &             &             &             \\[2mm] \hline 
39 &   9 & 2.03E-09 &  22 & 0 &  1.0000E+00 & -1.0000E+00 & -1.0000E+00 &  1.2511E-11 \\ 
   &     &          &     &   &  1.0000E+00 &             & -1.0000E+00 &             \\ 
   &     &          &     &   &  1.1100E-10 &             &             &             \\ 
   &     &          &     &   &  1.6069E-15 &             &             &             \\[2mm] \hline 
40 &   6 & 6.18E-09 &  11 & 0 &  7.9370E-01 & -2.5000E-01 &  5.0000E-01 &  2.2698E-11 \\ 
   &     &          &     &   &  7.0711E-01 &             & -4.7194E-01 &             \\ 
   &     &          &     &   &  5.2973E-01 &             &  3.5355E-01 &             \\ 
   &     &          &     &   &  8.4090E-01 &             &             &             \\[2mm] \hline 
42 &   7 & 1.23E-09 &  14 & 0 &  2.0000E+00 &  1.3858E+01 & -2.0000E+00 &  6.0051E-12 \\ 
   &     &          &     &   &  2.0000E+00 &             &  2.5355E+00 &             \\ 
   &     &          &     &   &  8.4853E-01 &             &             &             \\ 
   &     &          &     &   &  1.1314E+00 &             &             &             \\[2mm] \hline 
47 &   6 & 1.61E-09 &  25 & 0 &  1.0000E+00 &  2.2413E-20 & -1.7803E-11 &  4.3042E-13 \\ 
   &     &          &     &   &  1.0000E+00 &             & -4.7435E-11 &             \\ 
   &     &          &     &   &  1.0000E+00 &             &  1.1626E-10 &             \\ 
   &     &          &     &   &  1.0000E+00 &             &             &             \\ 
   &     &          &     &   &  1.0000E+00 &             &             &             \\[2mm] \hline 
48 &   4 & 2.91E-09 &   6 & 0 &  1.0000E+00 &  2.9251E-21 & -9.0754E-12 &  9.4296E-13 \\ 
   &     &          &     &   &  1.0000E+00 &             & -1.3989E-11 &             \\ 
   &     &          &     &   &  1.0000E+00 &             &             &             \\ 
   &     &          &     &   &  1.0000E+00 &             &             &             \\ 
   &     &          &     &   &  1.0000E+00 &             &             &             \\[2mm] \hline 
49 &  16 & 7.48E-09 &  58 & 0 &  1.0018E+00 &  6.2683E-13 &  5.7117E-10 &  6.2172E-15 \\ 
   &     &          &     &   &  1.0018E+00 &             &  8.2825E-11 &             \\ 
   &     &          &     &   &  1.0000E+00 &             &             &             \\ 
   &     &          &     &   &  9.9911E-01 &             &             &             \\ 
   &     &          &     &   &  1.0000E+00 &             &             &             \\[2mm] \hline 
50 &   4 & 1.08E-09 &  11 & 0 &  1.0000E+00 &  8.0396E-21 & -2.1761E-11 &  3.3669E-13 \\ 
   &     &          &     &   &  1.0000E+00 &             &  1.0586E-11 &             \\ 
   &     &          &     &   &  1.0000E+00 &             & -1.2536E-11 &             \\ 
   &     &          &     &   &  1.0000E+00 &             &             &             \\ 
   &     &          &     &   &  1.0000E+00 &             &             &             \\[2mm] \hline 
51 &   6 & 2.24E-09 &   8 & 0 &  1.0000E+00 &  5.6697E-24 & -2.2413E-10 &  1.3496E-12 \\ 
   &     &          &     &   &  1.0000E+00 &             & -1.9609E-10 &             \\ 
   &     &          &     &   &  1.0000E+00 &             &  6.8389E-10 &             \\ 
   &     &          &     &   &  1.0000E+00 &             &             &             \\ 
   &     &          &     &   &  1.0000E+00 &             &             &             \\[2mm] \hline 
52 &  10 & 6.35E-10 &  22 & 0 & -9.4556E-02 &  5.3266E+00 &  3.2779E+00 &  2.1715E-12 \\ 
   &     &          &     &   &  3.1519E-02 &             &  2.9054E+00 &             \\ 
   &     &          &     &   &  5.1576E-01 &             & -7.7479E+00 &             \\ 
   &     &          &     &   & -4.5272E-01 &             &             &             \\ 
   &     &          &     &   &  3.1519E-02 &             &             &             \\[2mm] \hline 
56 & 100 & 1.40E+91 & 298 & 1 &  1.5152E+00 &  1.0502E+00 & -5.0000E+03 &  6.3584E+00 \\ 
   &     &          &     &   &  9.6954E-01 &             & -5.0000E+03 &             \\ 
   &     &          &     &   & -7.1490E-01 &             & -5.0000E+03 &             \\ 
   &     &          &     &   &  2.9998E+19 &             &  5.0000E+03 &             \\ 
   &     &          &     &   &  7.1635E+18 &             &             &             \\ 
   &     &          &     &   & -2.2724E+19 &             &             &             \\ 
   &     &          &     &   & -1.9009E+19 &             &             &             \\[2mm] \hline 
61 &   6 & 1.49E-09 &  16 & 0 &  5.3268E+00 & -1.4365E+02 & -8.8768E-01 &  1.0668E-12 \\ 
   &     &          &     &   & -2.1190E+00 &             & -1.7378E+00 &             \\ 
   &     &          &     &   &  3.2105E+00 &             &             &             \\[2mm] \hline 
77 &   7 & 5.21E-10 &  21 & 0 &  1.1662E+00 &  2.4151E-01 & -8.5540E-02 &  4.8630E-14 \\ 
   &     &          &     &   &  1.1821E+00 &             & -3.1878E-02 &             \\ 
   &     &          &     &   &  1.3803E+00 &             &             &             \\ 
   &     &          &     &   &  1.5060E+00 &             &             &             \\ 
   &     &          &     &   &  6.1092E-01 &             &             &             \\[2mm] \hline 
78 &   6 & 1.13E-09 &  14 & 0 & -1.7171E+00 & -2.9197E+00 &  7.4445E-01 &  2.5413E-13 \\ 
   &     &          &     &   &  1.5957E+00 &             & -7.0358E-01 &             \\ 
   &     &          &     &   &  1.8272E+00 &             &  9.6806E-02 &             \\ 
   &     &          &     &   & -7.6364E-01 &             &             &             \\ 
   &     &          &     &   & -7.6364E-01 &             &             &             \\[2mm] \hline 
79 & 100 & 2.47E+06 & 375 & 1 &  5.7299E-01 &  5.9982E+01 &  5.0000E+03 &  2.3441E-12 \\ 
   &     &          &     &   &  1.1160E+00 &             & -5.0000E+03 &             \\ 
   &     &          &     &   &  1.6417E+00 &             &  5.0000E+03 &             \\ 
   &     &          &     &   &  4.4075E+00 &             &             &             \\ 
   &     &          &     &   &  3.4904E+00 &             &             &             \\[2mm] \hline 
501 &   7 & 1.90E-10 &  16 & 0 &  1.0000E+00 & -1.5000E+00 &  5.5556E-01 &  5.1514E-13 \\[2mm] \hline 
502 &   4 & 9.17E-09 &   3 & 0 &  5.9086E-11 &  1.7456E-21 & -5.9073E-11 &  5.9086E-11 \\[2mm] \hline 
503 &   4 & 9.46E-09 &   3 & 0 & -2.1551E-11 &  9.2889E-22 &  4.3099E-11 &  4.3102E-11 \\ 
    &     &          &     &   & -2.1551E-11 &             &             &             \\[2mm] \hline 
504 &   5 & 5.47E-11 &  20 & 0 &  1.0000E+00 &  0.0000E+00 & -6.4788E-14 &  5.2625E-14 \\[2mm] \hline 
505 &   9 & 1.51E-09 &  51 & 0 &  1.1604E-11 & -1.0000E+00 &  1.5000E+00 &  3.3107E-13 \\ 
    &     &          &     &   & -1.0000E+00 &             &             &             \\ 
    &     &          &     &   &  2.4703E-14 &             &             &             \\[2mm] \hline 
506 &   6 & 7.71E-09 &  19 & 0 & -7.0711E-01 & -1.4142E+00 &  7.0711E-01 &  2.0228E-11 \\ 
    &     &          &     &   & -7.0711E-01 &             &             &             \\[2mm] \hline 
507 &   6 & 1.73E-09 &  13 & 0 & -1.0000E+00 & -1.0000E+00 & -5.0000E-01 &  4.5453E-12 \\[2mm] \hline 
508 &   6 & 6.77E-10 &  45 & 0 &  1.0000E+00 &  6.2476E-24 & -4.9880E-12 &  2.5120E-12 \\ 
    &     &          &     &   &  1.0000E+00 &             &             &             \\[2mm] \hline 
509 &   5 & 1.65E-09 &  17 & 0 &  6.0000E+00 & -1.0800E+02 &  1.5000E+00 &  2.1316E-13 \\ 
    &     &          &     &   &  3.0000E+00 &             &             &             \\[2mm] \hline 
510 &   7 & 4.37E-10 &  27 & 0 & -5.3452E-01 & -3.7417E+00 &  1.8708E+00 &  9.7899E-13 \\ 
    &     &          &     &   & -8.0178E-01 &             &             &             \\ 
    &     &          &     &   & -2.6726E-01 &             &             &             \\[2mm] \hline 
511 & 100 & 2.07E+15 & 501 & 1 &  5.1787E-11 &  1.5724E-05 &  5.0000E+03 &  1.4917E-10 \\ 
    &     &          &     &   &  1.5724E-05 &             &  5.0000E+03 &             \\[2mm] \hline 
512 &   6 & 6.19E-09 &   7 & 0 & -7.0711E-01 & -9.8777E-01 &  1.1027E-01 &  1.6256E-11 \\ 
    &     &          &     &   & -7.0711E-01 &             &             &             \\[2mm] \hline 
513 &   2 & 6.45E-10 &   3 & 0 & -4.5608E-12 & -4.3266E-46 & -3.2249E-10 &  4.5608E-12 \\[2mm] \hline 
514 &   7 & 3.55E-10 &   6 & 0 &  1.0000E+00 &  5.0000E-01 & -1.0000E+00 &  2.1860E-13 \\ 
    &     &          &     &   &  3.5348E-10 &             &             &             \\[2mm] \hline 
  
\caption{Problems solved by Algorithm 2} 
\label{table:alg2} \\ 
\end{longtable} 
\end{landscape}
}

%TABLA ALGORITMO 3
{\scriptsize
\begin{landscape} 
 
\begin{longtable}{>{\centering\arraybackslash}m{1.0cm}|>{\centering\arraybackslash}m{1.0cm}|>{\centering\arraybackslash}m{2.0cm}|>{\centering\arraybackslash}m{1.0cm}|>{\centering\arraybackslash}m{1.2cm}|>{\centering\arraybackslash}m{2.0cm}|>{\centering\arraybackslash}m{2.0cm}|>{\centering\arraybackslash}m{2.0cm}|>{\centering\arraybackslash}m{2.0cm}}
Prob. & It. & KKT norm & Int. it. & Inform & $x$ & $f \left( x \right)$ & $\lambda$ & Infeas. \\[2mm] \hline
\endfirsthead
Prob. & It. & KKT norm & Int. it. & Inform & $x$ & $f \left( x \right)$ & $\lambda$ & Infeas. \\[2mm] \hline
\endhead
\hline
\multicolumn{9}{r}{\textit{Continued on the next page.}}
\endfoot
\endlastfoot

6  &  13 & 2.57E-09 & 110 & 0 &  1.0000E+00 &  7.3651E-18 &  3.4826E-10 &  6.6613E-14 \\ 
   &     &          &     &   &  1.0000E+00 &             &             &             \\[2mm] \hline 
7  &   8 & 1.90E-09 &  46 & 0 &  5.0272E-10 & -1.7321E+00 &  2.8868E-01 &  1.4895E-12 \\ 
   &     &          &     &   &  1.7321E+00 &             &             &             \\[2mm] \hline 
8  &  13 & 5.35E-09 &  40 & 0 &  4.6016E+00 & -1.0000E+00 &  1.0027E-10 &  2.3202E-13 \\ 
   &     &          &     &   &  1.9558E+00 &             & -3.3556E-11 &             \\[2mm] \hline 
9  &   5 & 2.18E-09 &  14 & 0 & -1.5000E+01 & -5.0000E-01 & -3.2725E-02 &  1.4211E-13 \\ 
   &     &          &     &   & -2.0000E+01 &             &             &             \\[2mm] \hline 
26 & 100 & 4.74E+01 & 660 & 1 &  9.9437E-01 &  1.5625E-08 & -1.1138E-02 &  2.0828E-13 \\ 
   &     &          &     &   &  9.9438E-01 &             &             &             \\ 
   &     &          &     &   &  1.0056E+00 &             &             &             \\[2mm] \hline 
27 &   6 & 5.22E-09 &  47 & 0 & -1.0000E+00 &  4.0000E-02 &  4.0000E-02 &  2.9801E-11 \\ 
   &     &          &     &   &  1.0000E+00 &             &             &             \\ 
   &     &          &     &   & -2.0599E-09 &             &             &             \\[2mm] \hline 
28 &   7 & 2.63E-09 &  22 & 0 &  5.0000E-01 &  2.0302E-18 &  1.1782E-09 &  2.1871E-14 \\ 
   &     &          &     &   & -5.0000E-01 &             &             &             \\ 
   &     &          &     &   &  5.0000E-01 &             &             &             \\[2mm] \hline 
39 &   9 & 6.28E-09 &  36 & 0 &  1.0000E+00 & -1.0000E+00 & -1.0000E+00 &  8.5081E-11 \\ 
   &     &          &     &   &  1.0000E+00 &             & -1.0000E+00 &             \\ 
   &     &          &     &   &  2.7535E-10 &             &             &             \\ 
   &     &          &     &   &  4.2486E-11 &             &             &             \\[2mm] \hline 
40 &   6 & 9.23E-09 &  18 & 0 &  7.9370E-01 & -2.5000E-01 &  5.0000E-01 &  2.0256E-11 \\ 
   &     &          &     &   &  7.0711E-01 &             & -4.7194E-01 &             \\ 
   &     &          &     &   &  5.2973E-01 &             &  3.5355E-01 &             \\ 
   &     &          &     &   &  8.4090E-01 &             &             &             \\[2mm] \hline 
42 &   7 & 4.02E-09 &  22 & 0 &  2.0000E+00 &  1.3858E+01 & -2.0000E+00 &  8.7150E-12 \\ 
   &     &          &     &   &  2.0000E+00 &             &  2.5355E+00 &             \\ 
   &     &          &     &   &  8.4853E-01 &             &             &             \\ 
   &     &          &     &   &  1.1314E+00 &             &             &             \\[2mm] \hline 
47 &  10 & 9.57E-09 &  46 & 0 &  1.0000E+00 &  4.1159E-19 &  4.3795E-10 &  1.6206E-12 \\ 
   &     &          &     &   &  1.0000E+00 &             & -1.0239E-09 &             \\ 
   &     &          &     &   &  1.0000E+00 &             &  3.6878E-09 &             \\ 
   &     &          &     &   &  1.0000E+00 &             &             &             \\ 
   &     &          &     &   &  1.0000E+00 &             &             &             \\[2mm] \hline 
48 &   7 & 5.11E-09 &  22 & 0 &  1.0000E+00 &  6.1576E-18 &  1.2104E-09 &  4.9714E-14 \\ 
   &     &          &     &   &  1.0000E+00 &             &  4.8590E-10 &             \\ 
   &     &          &     &   &  1.0000E+00 &             &             &             \\ 
   &     &          &     &   &  1.0000E+00 &             &             &             \\ 
   &     &          &     &   &  1.0000E+00 &             &             &             \\[2mm] \hline 
49 &  24 & 6.57E-09 & 112 & 0 &  1.0022E+00 &  1.4621E-12 &  1.3366E-09 &  2.8087E-15 \\ 
   &     &          &     &   &  1.0022E+00 &             & -3.6112E-10 &             \\ 
   &     &          &     &   &  1.0000E+00 &             &             &             \\ 
   &     &          &     &   &  9.9890E-01 &             &             &             \\ 
   &     &          &     &   &  1.0000E+00 &             &             &             \\[2mm] \hline 
50 &  19 & 9.91E-09 &  63 & 0 &  1.0000E+00 &  1.3296E-17 & -6.0127E-10 &  4.0116E-14 \\ 
   &     &          &     &   &  1.0000E+00 &             &  4.0573E-10 &             \\ 
   &     &          &     &   &  1.0000E+00 &             & -6.9435E-10 &             \\ 
   &     &          &     &   &  1.0000E+00 &             &             &             \\ 
   &     &          &     &   &  1.0000E+00 &             &             &             \\[2mm] \hline 
51 &  13 & 7.98E-09 &  40 & 0 &  1.0000E+00 &  4.1029E-18 &  1.5427E-09 &  1.8011E-13 \\ 
   &     &          &     &   &  1.0000E+00 &             & -4.4049E-10 &             \\ 
   &     &          &     &   &  1.0000E+00 &             &  1.8594E-09 &             \\ 
   &     &          &     &   &  1.0000E+00 &             &             &             \\ 
   &     &          &     &   &  1.0000E+00 &             &             &             \\[2mm] \hline 
52 &  10 & 8.71E-09 &  32 & 0 & -9.4556E-02 &  5.3266E+00 &  3.2779E+00 &  2.5459E-11 \\ 
   &     &          &     &   &  3.1519E-02 &             &  2.9054E+00 &             \\ 
   &     &          &     &   &  5.1576E-01 &             & -7.7479E+00 &             \\ 
   &     &          &     &   & -4.5272E-01 &             &             &             \\ 
   &     &          &     &   &  3.1519E-02 &             &             &             \\[2mm] \hline 
56 & 100 & 5.80E+98 & 237 & 1 & -3.0520E+19 & -2.2250E+59 & -5.0000E+03 &  1.2964E+20 \\ 
   &     &          &     &   & -7.2904E+19 &             & -5.0000E+03 &             \\ 
   &     &          &     &   &  1.0000E+20 &             &  5.0000E+03 &             \\ 
   &     &          &     &   & -1.0000E+20 &             &  5.0000E+03 &             \\ 
   &     &          &     &   &  1.0000E+20 &             &             &             \\ 
   &     &          &     &   & -1.0000E+20 &             &             &             \\ 
   &     &          &     &   & -1.0000E+20 &             &             &             \\[2mm] \hline 
61 &   8 & 9.38E-09 &  37 & 0 &  5.3268E+00 & -1.4365E+02 & -8.8768E-01 &  3.3364E-12 \\ 
   &     &          &     &   & -2.1190E+00 &             & -1.7378E+00 &             \\ 
   &     &          &     &   &  3.2105E+00 &             &             &             \\[2mm] \hline 
77 &   8 & 9.97E-09 &  51 & 0 &  1.1662E+00 &  2.4151E-01 & -8.5540E-02 &  6.9963E-12 \\ 
   &     &          &     &   &  1.1821E+00 &             & -3.1878E-02 &             \\ 
   &     &          &     &   &  1.3803E+00 &             &             &             \\ 
   &     &          &     &   &  1.5060E+00 &             &             &             \\ 
   &     &          &     &   &  6.1092E-01 &             &             &             \\[2mm] \hline 
78 &  11 & 2.19E-09 &  36 & 0 & -1.7171E+00 & -2.9197E+00 &  7.4445E-01 &  5.9653E-13 \\ 
   &     &          &     &   &  1.5957E+00 &             & -7.0358E-01 &             \\ 
   &     &          &     &   &  1.8272E+00 &             &  9.6806E-02 &             \\ 
   &     &          &     &   & -7.6364E-01 &             &             &             \\ 
   &     &          &     &   & -7.6364E-01 &             &             &             \\[2mm] \hline 
79 & 100 & 2.24E+20 & 330 & 1 & -3.2497E+01 &  6.0361E+08 & -5.0000E+03 &  3.2945E+06 \\ 
   &     &          &     &   & -5.3327E+01 &             &  5.0000E+03 &             \\ 
   &     &          &     &   &  1.4875E+02 &             &  5.0000E+03 &             \\ 
   &     &          &     &   &  1.7034E+02 &             &             &             \\ 
   &     &          &     &   &  1.3614E+01 &             &             &             \\[2mm] \hline 
501 &   6 & 1.03E-09 &  21 & 0 &  1.0000E+00 & -1.5000E+00 &  5.5556E-01 &  5.1683E-12 \\[2mm] \hline 
502 &   6 & 1.11E-11 &  21 & 0 &  6.3728E-15 &  2.0307E-29 & -7.5249E-15 &  6.3728E-15 \\[2mm] \hline 
503 &   6 & 1.60E-09 &  21 & 0 & -5.3100E-12 &  5.6392E-23 &  7.3257E-10 &  1.0620E-11 \\ 
    &     &          &     &   & -5.3100E-12 &             &             &             \\[2mm] \hline 
504 &   8 & 5.24E-09 &  40 & 0 &  2.0000E+00 &  9.0000E+00 & -1.0000E+00 &  4.2739E-12 \\[2mm] \hline 
505 &   7 & 6.66E-09 &  33 & 0 &  1.0000E+00 &  7.5269E-14 &  3.9434E-10 &  6.0957E-12 \\ 
    &     &          &     &   &  4.2221E-05 &             &             &             \\ 
    &     &          &     &   & -1.8313E-09 &             &             &             \\[2mm] \hline 
506 &   6 & 7.39E-09 &  22 & 0 &  7.0711E-01 &  1.4142E+00 & -7.0711E-01 &  1.0221E-11 \\ 
    &     &          &     &   &  7.0711E-01 &             &             &             \\[2mm] \hline 
507 &   6 & 5.76E-09 &  20 & 0 & -1.0000E+00 & -1.0000E+00 & -5.0000E-01 &  5.9488E-12 \\[2mm] \hline 
508 &   7 & 8.53E-09 &  51 & 0 &  1.0000E+00 &  9.4913E-20 & -1.0074E-08 &  2.2686E-12 \\ 
    &     &          &     &   &  1.0000E+00 &             &             &             \\[2mm] \hline 
509 &  11 & 9.62E-09 &  63 & 0 &  6.0000E+00 & -1.0800E+02 &  1.5000E+00 &  1.4779E-12 \\ 
    &     &          &     &   &  3.0000E+00 &             &             &             \\[2mm] \hline 
510 &   7 & 4.22E-09 &  42 & 0 & -5.3452E-01 & -3.7417E+00 &  1.8708E+00 &  9.4080E-13 \\ 
    &     &          &     &   & -8.0178E-01 &             &             &             \\ 
    &     &          &     &   & -2.6726E-01 &             &             &             \\[2mm] \hline 
511 & 100 & 4.59E-03 & 661 & 1 &  3.5683E-10 & -3.5009E-05 &  5.0000E+03 &  5.5027E-10 \\ 
    &     &          &     &   & -3.5009E-05 &             & -5.0000E+03 &             \\[2mm] \hline 
512 &   6 & 2.13E-09 &  19 & 0 & -7.0711E-01 & -9.8777E-01 &  1.1027E-01 &  1.9461E-12 \\ 
    &     &          &     &   & -7.0711E-01 &             &             &             \\[2mm] \hline 
513 &   5 & 5.38E-12 &  16 & 0 &  3.4568E-14 & -1.4280E-54 &  1.8983E-14 &  3.4568E-14 \\[2mm] \hline 
514 &   6 & 9.78E-09 &  19 & 0 &  1.0000E+00 &  5.0000E-01 & -1.0000E+00 &  5.5404E-11 \\ 
    &     &          &     &   & -1.9922E-10 &             &             &             \\[2mm] \hline 
 
\caption{Problems solved by Algorithm 3} 
\label{table:alg3} \\  
\end{longtable} 
\end{landscape}
}

\section{Conclusions} \label{sec:conclusions}
This paper presents a novel method for solving nonlinear programming problems based on the sharp augmented Lagrangian. It introduces a smoothed function to overcome the nondifferentiability of the sharp augmented Lagrangian, thereby facilitating the minimization process.

The exact algorithm presented in Section \ref{sec:exact} converges to a global solution of the primal problem. However, its practical implementation faces several challenges:
\begin{itemize}
\item Nondifferentiability: the objective function becomes nondifferentiable when the smoothing parameter $t$ approaches zero, potentially leading to numerical difficulties with optimization solvers that assume smoothness.
\item Global optimization: finding an exact global minimizer of the smoothed function is computationally expensive and can be challenging, particularly for large--scale problems.
\end{itemize}
These challenges demand relaxing the requirement for exact global minimizers.

The proposed inexact algorithms offer a more practical approach. A barrier function ensures the existence of inexact stationary points, even if the original function may have nondifferentiable points. One algorithm employs a fixed smoothing parameter, while the other uses a varying smoothing parameter. The convergence properties of both algorithms are analyzed, demonstrating that limit points of the generated sequences are either stationary points of the original problem or stationary points of a feasibility problem.

The boundedness of the penalty parameter is guaranteed under specific assumptions, including the convergence of the iterates, feasibility of the limit point, linear independence of the gradients, and satisfaction of the second--order sufficient optimality condition. Sufficient conditions are provided to ensure the boundedness of the penalty parameter, focusing on the relationship between the smoothing parameter and the constraint violation.

The proposed algorithms exhibit competitive performance compared to ALGENCAN, particularly in terms of the number of internal iterations. However, their performance varies depending on the specific problem characteristics, suggesting the need for further analysis and potential improvements. The promising results obtained in this study motivate further research to explore the potential of these algorithms for solving a wider range of problems.

\appendix

\section{Description of the problems 501--514}

\subsection*{\bf Problem 501}
\begin{tabular}{ll}
{\bf Objective function:} & $f \left( x \right) = \frac{1}{2} x^2 - 2 x .$ \\[2mm]
{\bf Constraints:} & $h \left( x \right) = x \left( x - 1 \right) \left( x + 1 \right) = 0 .$ \\[2mm]
{\bf Initial point:} & $x^0 = 2 .$ \\[2mm]
{\bf Solutions:} & $x^* = 1 , \quad \lambda^* = \frac{1}{2} .$
\end{tabular}

\subsection*{\bf Problem 502}
\begin{tabular}{ll}
{\bf Objective function:} & $f \left( x \right) = \frac{1}{2}x^2 .$ \\[2mm]
{\bf Constraints:} & $h \left( x \right) = x = 0 .$ \\[2mm]
{\bf Initial point:} & $x^0 = 10 .$ \\[2mm]
{\bf Solutions:} & $x^* = 0 , \quad \lambda^*=0 .$
\end{tabular}

\subsection*{\bf Problem 503}
\begin{tabular}{ll}
{\bf Objective function:} & $f \left( x \right) = x_1^2 + x_2^2 .$ \\[2mm]
{\bf Constraints:} & $h \left( x \right) = x_1 + x_2 = 0 .$ \\[2mm]
{\bf Initial point:} & $x^0 = \left( 3 , 3 \right) .$ \\[2mm]
{\bf Solutions:} & $x^* = \left( 0 , 0 \right) , \quad \lambda^* = 0 .$
\end{tabular}

\subsection*{\bf Problem 504}
\begin{tabular}{ll}
{\bf Objective function:} & $f \left( x \right) = \left( x^2 - 1 \right)^2 .$ \\[2mm]
{\bf Constraints:} & $h \left( x \right) = \left( x^2 - 1 \right) \left( x^2 - 4 \right) = 0 .$ \\[2mm]
{\bf Initial point:} & $x^0 = 10 .$ \\[2mm]
{\bf Solutions:} & $x^* = \pm 1 , \quad \lambda^* = 0 .$
\end{tabular}

\subsection*{\bf Problem 505}
\begin{tabular}{ll}
{\bf Objective function:} & $f \left( x \right) = x_2^3 + x_1x_3^2 .$ \\[2mm]
{\bf Constraints:} & $h \left( x \right) = x_1^2 + x_2^2 + x_3^2 - 1 = 0 .$ \\[2mm]
{\bf Initial point:} & $x^0 = \left( 1 , 1 , 1 \right) .$ \\[2mm]
{\bf Solutions:} & $x^* = \left( 0 , - 1 , 0 \right) , \quad \lambda^* = \frac{3}{2} .$
\end{tabular}

\subsection*{\bf Problem 506}
\begin{tabular}{ll}
{\bf Objective function:} & $f \left( x \right) = x_1 + x_2 .$ \\[2mm]
{\bf Constraints:} & $h \left( x \right) = x_1^2 + x_2^2 - 1 = 0 .$ \\[2mm]
{\bf Initial point:} & $x^0 = \left( 10 , 10 \right) .$ \\[2mm]
{\bf Solutions:} & $x^* = \left( - \frac{\sqrt{2}}{2} , - \frac{\sqrt{2}}{2} \right) , \quad \lambda^* = \frac{\sqrt{2}}{2} .$
\end{tabular}

\subsection*{\bf Problem 507}
\begin{tabular}{ll}
{\bf Objective function:} & $f \left( x \right) = x .$ \\[2mm]
{\bf Constraints:} & $h \left( x \right) = x^3 - x = 0 .$ \\[2mm]
{\bf Initial point:} & $x^0 = - 1.5 .$ \\[2mm]
{\bf Solutions:} & $x^* = - 1 , \quad \lambda^* = - \frac{1}{2} .$
\end{tabular}

\subsection*{\bf Problem 508}
\begin{tabular}{ll}
{\bf Objective function:} & $f \left( x \right) = 100 \left( x_2 - x_1^2 \right)^2 + \left( 1 - x_1 \right)^2 .$ \\[2mm]
{\bf Constraints:} & $h \left( x \right) = x_1 - x_2 = 0 .$ \\[2mm]
{\bf Initial point:} & $x^0 = \left( 100 , 1.2 \right) .$ \\[2mm]
{\bf Solutions:} & $x^* = \left( 1 , 1 \right) , \quad \lambda^* = 0 .$
\end{tabular}

\subsection*{\bf Problem 509}
\begin{tabular}{ll}
{\bf Objective function:} & $f \left( x \right) = - x_1^2 x_2 .$ \\[2mm]
{\bf Constraints:} & $h \left( x \right) = 4 x_1 x_2 + x_1^2 - 108 = 0 .$ \\[2mm]
{\bf Initial point:} & $x^0 = \left( 3 , 3 \right) .$ \\[2mm]
{\bf Solutions:} & $x^* = \left( 6 , 3 \right) , \quad \lambda^* = \frac{3}{2} .$
\end{tabular}

\subsection*{\bf Problem 510}
\begin{tabular}{ll}
{\bf Objective function:} & $f \left( x \right) = 2 x_1 + 3 x_2 + x_3 .$ \\[2mm]
{\bf Constraints:} & $h \left( x \right) = x_1^2 + x_2^2 + x_3^2 - 1 = 0 .$ \\[2mm]
{\bf Initial point:} & $x^0 = \left( 1 , 1 , 1 \right) .$ \\[2mm]
{\bf Solutions:} & $x^* = \frac{- \sqrt{14}}{14} \left( 2 , 3 , 1 \right) , \quad \lambda^* = \frac{\sqrt{14}}{2} .$
\end{tabular}

\subsection*{\bf Problem 511}
\begin{tabular}{ll}
{\bf Objective function:} & $f \left( x \right) = x_1 + x_2 .$ \\[2mm]
{\bf Constraints:} & $h_1 \left( x \right) = \left( x_1 - 1 \right)^2 + x_2^2 - 1 = 0 ,$ \\[2mm]
 & $h_2 \left( x \right) = \left( x_1 - 2 \right)^2 + x_2^2 - 4 = 0 .$ \\[2mm]
{\bf Initial point:} & $x^0 = \left( 1 , 1 \right) .$ \\[2mm]
{\bf Solutions:} & $x^* = \left( 0 , 0 \right)$ (the only feasible point).
\end{tabular}

\subsection*{\bf Problem 512}
\begin{tabular}{ll}
{\bf Objective function:} & $f \left( x \right) = \sin \left( x_1 + x_2 \right) .$ \\[2mm]
{\bf Constraints:} & $h \left( x \right) = x_1^2 + x_2^2 - 1 = 0 .$ \\[2mm]
{\bf Initial point:} & $x^0 = \left( 0 , 0 \right) .$ \\[2mm]
{\bf Solutions:} & $x^* = \left( - \frac{\sqrt{2}}{2} , - \frac{\sqrt{2}}{2} \right) , \quad \lambda^* = \frac{\sqrt{2}}{2} \cos \left( \sqrt{2} \right) .$
\end{tabular}

\subsection*{\bf Problem 513}
\begin{tabular}{ll}
{\bf Objective function:} & $f \left( x \right) = - x^4 .$ \\[2mm]
{\bf Constraints:} & $h \left( x \right) = x = 0 .$ \\[2mm]
{\bf Initial point:} & $x^0 = 1 .$ \\[2mm]
{\bf Solutions:} & $x^* = 0 , \quad \lambda^* = 0 .$
\end{tabular}

\subsection*{\bf Problem 514}
%$\scriptstyle 514$ & $\scriptstyle \frac{1}{2}(x_1^2 + x_2^2)$ & $\scriptstyle x_1-1$ & $\scriptstyle (4.9,0.1)$ & $\scriptstyle x^*=(1,0)$ & \\ \hline
\begin{tabular}{ll}
{\bf Objective function:} & $f \left( x \right) = \frac{1}{2} \left( x_1^2 + x_2^2 \right) .$ \\[2mm]
{\bf Constraints:} & $h \left( x \right) = x_1 - 1 = 0 .$ \\[2mm]
{\bf Initial point:} & $x^0 = \left( 4.9 , 0.1 \right) .$ \\[2mm]
{\bf Solutions:} & $x^* = \left( 1 , 0 \right) , \quad \lambda^* = - 1 .$
\end{tabular}

\begin{acknowledgements}
This work was partially supported by the following grants from SGCyT--UNNE (19F010), ANPCyT (PICT-2021-GRF-TI-00188, PICT-2019-2019-04569) and SeCYT--UNC (33620230100671CB).
\end{acknowledgements}

%References
%\printbibliography
\bibliographystyle{plain}%{abbrv}%{spmpsci}
\bibliography{referencias}

\end{document}